\documentclass[10pt,a4paper,reqno]{amsproc}
\usepackage{longtable,cite}
\usepackage{lscape}
\usepackage{multirow}
\usepackage[pdftex,bookmarks=true]{hyperref}
\newtheorem{theorem}{Theorem}[section]
\newtheorem{lemma}[theorem]{Lemma}

\newtheorem{corollary}[theorem]{Corollary}
\theoremstyle{definition}
\newtheorem{definition}[theorem]{Definition}

\theoremstyle{remark}

\numberwithin{equation}{section}
\usepackage{tikz}
\usetikzlibrary{shapes, arrows}
\usetikzlibrary{chains,arrows.meta}

\begin{document}
	\title [The Multidimensional QPFT: Theoretical Analysis and Applications]{The Multidimensional Quadratic Phase Fourier Transform: Theoretical Analysis and Applications} 
	\author{Sarga Varghese}
	\address{Department of Mathematics, SRM University AP, Amaravati-522240, India}
	
	\email{sargavarghese$@$gmail.com, gitamahato1158$@$gmail.com and manabiitism17$@$gmail.com}
	
	\author{Gita Rani Mahato}
	
	\author{Manab Kundu}
	\date{}
	
	
	\maketitle              
	
	\begin{abstract}		
		The quadratic phase Fourier transform (QPFT) is a generalization of several well known integral transforms, including the linear canonical transform (LCT), fractional Fourier transform (FrFT), and Fourier transform (FT). This paper introduces the multidimensional QPFT and investigates its theoretical properties, including Parseval’s identity and inversion theorems. Generalized convolutions and correlation for multiple variables, extending the conventional convolution for single-variable functions, is proposed within the QPFT setting. Additionally, Boas type theorem for the multidimensional QPFT is established. As applications, multiplicative filter design and the solution of integral equations using the proposed convolution operation are explored.

	\end{abstract}

	\section{Introduction}
	Integral transforms \cite{intt, intt1} have been used for nearly two centuries to solve various problems in applied mathematics, mathematical physics, and engineering science. By converting functions from one domain to another, integral transforms simplify complex problems, enabling the analysis and solution of differential equations, and the study of signal properties. The Fourier transform (FT), one of the widely used integral transform, is a significant tool used in signal processing and harmonic analysis \cite{1}.  Despite its long history, the field continues to evolve, with new transforms like the fractional Fourier transform(FrFT), linear canonical transform (LCT), and wavelet transform emerging in the past three decades.
	In 1980, Victor Namias \cite{frac1} introduced the fractional Fourier transform (FrFT), which he used to solve problems in quantum mechanics. By the 1990s, it gained attention for its applications in signal processing and optics. The FrFT generalizes the FT by introducing a fractional order parameter, allowing the transform to rotate signals in the time-frequency plane \cite{frac, frac2, frac3, frac4}. The FrFT was introduced to address the limitations of the FT in analyzing signals with time-varying frequency content. It is a special case of the LCT, a broader class of integral transforms that extends FrFT by incorporating four transformation parameters. The LCT, in turn, allows for more flexible time-frequency analysis \cite{lct, lct1, lct2, lct3}.
	
	Further extending these concepts, the quadratic phase Fourier transform (QPFT) includes quadratic phase terms to model complex signal behaviors, making it useful in applications such as radar, sonar, and optical systems. Also, it is used in linear time-invariant (LTI) system, signal analysis, solving integral equations, etc. QPFT was introduced by Castro et al. \cite{lpc, lmm} motivated by the work of Saitoh \cite{sai}, who used the reproducing kernel theory of QPFT for solving heat equation. Later, Prasad et al. \cite{ta,pb, sp1, sp2} discussed the QPFT framework by developing the wavelet transform, and studied its underlying features.  QPFT can be used to identify single and multi-component linear frequency modulated signals by introducing a modified ambiguity function and the Wigner distribution associated with QPFT \cite{tml}. For more recent works in QPFT theory, one can refer to \cite{yg, var, bhat}. In this paper, we are mainly focusing on
	(a) new convolution and correlation for LTI system and solving integral equations, (b) Boas theorem for analyzing the compactly supported signals, for the multidimensional cases. 
	
	The theoretical foundation of linear time-invariant (LTI) systems heavily relies on convolution and correlation processes. These methods work very well in a variety of fields, including pattern recognition, optics, and signal processing. Any continuous LTI structure's efficiency can be calculated by integrating the impulse response of the system with the incoming signals \cite{lb1}. In recent times, Wei et al. and Li et al. have made remarkable contributions in the field of convolution and correlation. They proposed new structures for several integral transforms and investigated their sampling theorems for band-limited signals \cite{dw4,dw3,dw2,dwy,dwq,dw1,bl1,bl2}.

	Boas' theorem \cite{boas} characterizes high-frequency signals by relating the vanishing of signal frequency content to a specific integral transform. This result is crucial for understanding high-frequency signals in Fourier analysis. It states that the signal frequency content \( f \) vanishes almost everywhere on the band \( (-1, 1) \) if and only if the Boas transform \( (Bf)(x) \) satisfies the equation:
	\begin{align*}
		(B(Bf))(x) = -f(x), \quad x \in \mathbb{R}.
	\end{align*}	
	Here, \( B \) is defined as:
	\begin{align*}
		(Bf)(x) = \frac{1}{\pi} \int_{0}^{\infty} \frac{f(x+t) - f(x-t)}{t^2} \sin(t) \, dt.
	\end{align*}
	In \cite{tuan},Tuan proposed another description of high frequency signals, which enables to describe band-pass and band-stop frequency signals fully.
	
	The integral equations, one of the most useful mathematical tools in both pure and applied mathematics, has enormous applications in many physical problems. In 1825 Abel, an Italian mathematician, first produced an integral equation in connection with the famous tautochrone problem. Singular integral equations encountered by Abel can easily be solved by using integral transforms. Also, many initial and boundary value problems associated with ordinary differential equation (ODE) and partial differential equation (PDE) can be transformed into problems of solving some approximate integral equations. The development of science has led to the formulation of physical laws that are often expressed as differential equations, which play a crucial role in solving engineering problems. These differential equations can be transformed into equivalent integral equations. Integral equations are equally important, as many physical problems are directly governed by them, and they can also be converted back into differential equations. Solving integral equations is essential for addressing complex physical and engineering challenges, providing a versatile framework for modeling and analyzing real-world phenomena. Their ability to bridge differential equations and practical applications highlights their significance in scientific and engineering research \cite{inte , inte1, inte2}.
	
	Multidimensional signals, depends on multiple independent variables such as space, time, or frequency, play crucial role in modern signal processing and systems theory. These signals are essential in modeling and analyzing complex systems where interactions across multiple dimensions are involved, such as in image processing, video compression, geophysical signal analysis, and multidimensional control systems. The mathematical framework of multidimensional systems involves functions and polynomials in several complex or real variables, which has been instrumental in advancing areas like iterative learning control, distributed network synthesis, and nonlinear system analysis through multidimensional transforms. Multidimensional signal processing has been traditionally based on the concepts and theory of linear systems and Fourier analysis (or other related transforms) \cite{multis0, multis1, multis2, multis3}.

	Recently, there have been significant advancements in the developement of multidimensional integral transforms. Zayed \cite{Zayed} (2018) introduced a new two-dimensional FrFT and established its properties. Building on this work, Roopkumar et al. \cite{roop} (2019) provided rigorous proofs of the inversion theorems for the multidimensional FrFT on $L^p(\mathbb{R}^N)$ for $p =1,2$, and derived the Parseval’s theorem for $L^2(\mathbb{R}^N)$. Additionally, they developed a novel approach to define convolution and correlation for the multidimensional FrFT. Later, Kundu et al. \cite{Kundu} (2022) extended the theory of the linear canonical transform to the multidimensional domain, establishing its convolution properties and further enriching the framework of multidimensional signal processing and Ahmad et al. \cite{ahm} studied uncertainty principles associated with multidimensional LCT.
	
	In \cite{castro}, Castro et al. introduced the multidimensional quadratic-phase Fourier transform (MQPFT), extending the one-dimensional QPFT to higher dimensions in non-separable sense, and established its properties, uncertainty principles, and applications. Kumar et al. \cite{kumar} further developed the n-dimensional pseudo-differential operator (n-DPDO) using the n-dimensional quadratic phase Fourier transform (n-DQPFT), proving its continuity and boundedness in \(L^p(\mathbb{R}^n)\) and applying it to generalized Schrödinger-type equations, enhancing its role in quantum mechanics. In \cite{var}, we have proved the convolution, correlation and Boas theorem for QPFT. Motivated by these developements, our work aims to further explore and extend the theory of QPFT to multidimensional domains. 
	
	The goals of this article are as follows:
	\begin{itemize}
		\item To introduce a new definition of the multidimensional QPFT. 
		\item To derive properties of the new transform, such as its inversion formula, Plancherel theorem, convolution theorem, and Boas theorem.
		\item To study the applications of the proposed convolution.
	\end{itemize}
	
	This paper is structured as:  Section 2 presents the basic concepts of FT, QPFT and multidimensional QPFT. Section 3 focuses on the convolution and correlation in the context of multidimensional QPFT, which includes three types of convolutions given in \ref{conv1}, \ref{conv2} and \ref{conv3}. Section 4 discusses inversion theorem, while Section 5 proves Boas theorem for multidimensional QPFT. Section 6 explores potential applications, including solving integral equations and designing multiplicative filters. Concluding remarks are provided in Section 7. 
	
	\section{Preliminaries} 
	This section offers a brief summary of the basic concepts and notations necessary for 
	\\For each $p \in [1, \infty)$, we have
	\begin{eqnarray*}
		L^p(\mathbb{R}^N) = \bigg\{ f~:~ ||f||_p = \Big( \frac{1}{\sqrt{(2\pi)}^N} \int_{\mathbb{R}^N} |f(\boldsymbol{x})|^p d\boldsymbol{x}\Big)^\frac{1}{p}\bigg\},
	\end{eqnarray*} 
	which is the space of all $p$-integrable functions on $\mathbb{R}^N$. 
	
	\begin{definition}
		For a function $f \in L^1(\mathbb{R}^N)$, the multidimensional FT of $f$ is defined by
		\begin{eqnarray*}
			\hat{ f}(\boldsymbol{x})= \frac{1}{(\sqrt{2\pi})^N} \int_{\mathbb{R}^N} f(\boldsymbol{t}) e^{-i\boldsymbol{t . x}} d\boldsymbol{t}, ~~ \boldsymbol{x} \in \mathbb{R}^N.
		\end{eqnarray*}
	\end{definition}

	\begin{definition}\label{def1}
		For a function $f \in L^1(\mathbb{R})$, the quadratic phase Fourier transform (QPFT) of $f$, for a given set of parameters $\Lambda=\{ a,b,c,d,e\}$, $b\neq 0$, is defined by \cite{ta}
		\begin{eqnarray}
			\mathcal{Q}_\Lambda(\omega)= \tilde{F}(\omega)=  \int_{\mathbb{R}} \mathcal{K}_\Lambda(\omega,x) f(x) dx, \hspace{3mm} \omega\in \mathbb{R},
		\end{eqnarray}
		where
		\begin{eqnarray}
			\mathcal{K}_\Lambda (\omega, x)=  \sqrt{\frac{b}{2\pi i}}~e^{i(a x^2 + b x \omega + c \omega^2 + d x + e\omega)}
		\end{eqnarray}
		is the kernel of QPFT and	$ a, b, c, d, e \in \mathbb{R}$.
		Then the inversion of QPFT is given by
		\begin{eqnarray}
			f(x)=  \Big(\mathcal{Q}_{-\Lambda} (\mathcal{Q}_{\Lambda} f)\Big)(x)=  \int_{\mathbb{R}}  \mathcal{K}_{-\Lambda} (\omega, x) (\mathcal{Q}_{\Lambda} f)(\omega) d\omega,
		\end{eqnarray}
		where
		\begin{eqnarray*}
			\mathcal{K}_{-\Lambda} (\omega, x)= \sqrt{\frac{bi}{2\pi}}~e^{-i(a x^2 + b x \omega + c \omega^2 + d x + e\omega)}.
		\end{eqnarray*}
	\end{definition}
	
	\begin{definition}\label{def}
		The multidimensional QPFT of a function $f \in L^1(\mathbb{R}^N)$ is defined as 
		\begin{eqnarray*}
			(\mathcal{Q}_{\boldsymbol{\Lambda}} f)(\boldsymbol{\omega})=  \int_{\mathbb{R}^N} \mathcal{K}_{\boldsymbol{\Lambda}}(\boldsymbol{\omega},\boldsymbol{x}) f(\boldsymbol{x}) d\boldsymbol{x},
		\end{eqnarray*}
		where
		$\boldsymbol{\omega}=(\omega_1, \omega_2, . . . , \omega_N), \boldsymbol{x}=(x_1, x_2, . . . , x_N)\in \mathbb{R}^N, 
		\mathcal{K}_{\boldsymbol{\Lambda}}(\boldsymbol{\omega},\boldsymbol{x})= \prod_{k=1}^{N} \mathcal{K}_{\Lambda_k}(\omega_k,x_k), \\\boldsymbol{\Lambda}=(\Lambda_1, \Lambda_2, . . . , \Lambda_N) \text{ and } \Lambda_k = (a_k, b_k, c_k, d_k, e_k),  k=1, 2, . . . , N.$
	\end{definition}
	Throughout this paper, we use bold letters to denote N-tuple, for instance, $\boldsymbol{a} = (a_1, a_2, . . . , a_N)$ . For $\lambda \in \mathbb{R}$, denote 
	\begin{eqnarray}\label{ep}
		e_\lambda^{\boldsymbol{a, d}}(\boldsymbol{x}) = e^{i \lambda \sum_{k=1}^{N} (a_k x_k^2 + d_k x_k)}
	\end{eqnarray}  for multidimensional domain.
	Obviously, $e_\lambda^{-\boldsymbol{a, d}}(\boldsymbol{x}) = e_{-\lambda}^{\boldsymbol{a, d}} (\boldsymbol{x})$ and then the multidimensional QPFT can be rewritten as
	\begin{eqnarray*}
		(\mathcal{Q}_{\boldsymbol{\Lambda}} f)(\boldsymbol{\omega})= C_{\boldsymbol{\Lambda}} e_1^{\boldsymbol{c, e}}(\boldsymbol{\omega})\widehat{(e_1^{\boldsymbol{a, d}}f(\boldsymbol{x}))} (\boldsymbol{b\omega}),
	\end{eqnarray*}
	where 
	\begin{eqnarray}\label{C}
		C_{\boldsymbol{\Lambda}} = \prod_{k=1}^{N} \sqrt{\frac{b_k}{i}}
	\end{eqnarray} 
	and $\widehat{(e_1^{\boldsymbol{a, d}}f(\boldsymbol{x}))}$ is the Fourier transform of $(e_1^{\boldsymbol{a, d}}f(\boldsymbol{x}))$. Using this notation, the kernel can be rewritten as
	\begin{eqnarray*}
		\mathcal{K}_{\boldsymbol{\Lambda}}(\boldsymbol{\omega},\boldsymbol{x}) = \frac{C_{\boldsymbol{\Lambda}}}{(2\pi)^\frac{N}{2}}  e_1^{\boldsymbol{a, d}}(\boldsymbol{x}) e_1^{\boldsymbol{c, e}}(\boldsymbol{\omega}) e^{i \sum_{k=1}^{N}b_k x_k \omega_k }.
	\end{eqnarray*}

	\begin{theorem}
		For functions $f,g \in L^2 (\mathbb{R}^N)$, the inner product of multidimensional QPFT of two functions $f$ and $g$ is same as the inner product of $f$ and $g$.
		\\i.e.,
		\begin{eqnarray*}
			\big< f, g \big> &=& \big< \mathcal{Q}_{\boldsymbol{\Lambda}} f, \mathcal{Q}_{\boldsymbol{\Lambda}} g\big>.
		\end{eqnarray*}	
	\end{theorem}
	As the proof is straightforward, we omit the proof part.
	\begin{corollary}\label{nm}
		If we take $f=g$, the equation takes the particular form as
		\begin{eqnarray}
			||f||_2^2 &=& ||\mathcal{Q}_{\boldsymbol{\Lambda}} f||_2^2. 
		\end{eqnarray}
	\end{corollary}

	\section{Convolution and correlation}
	In this Section, we discuss analogous type convolution and correlation theorems for multidimensional QPFT \cite{ta}.
	\begin{definition}
		For the functions $f, g \in L^1 (\mathbb{R})$, we can define convolution by \cite{ta}
		\begin{eqnarray}
			(f * g)(x) = f(x) * g(x) = \frac{1}{\sqrt{2 \pi}} \int_{\mathbb{R}} f(\tau) g(x-\tau) d\tau,
		\end{eqnarray}
		which satisfies
		\begin{eqnarray}
			\hat{(f * g)}(\omega) =  (\hat{ f})(\omega) (\hat{ g})(\omega),
		\end{eqnarray}
		where $\hat{ f}$ is Fourier transform of $f$.
	\end{definition}	
	Analogous to this, the convolution in multidimensional case is defined as follows.
	\begin{definition}
		For the functions $f, g \in L^1 (\mathbb{R}^N)$, we can define convolution by 
		\begin{eqnarray}
			(f \star g)(\boldsymbol{x}) = f(\boldsymbol{x}) \star g(\boldsymbol{x}) = \frac{1}{(\sqrt{2 \pi})^N} \int_{\mathbb{R}^N} f(\boldsymbol{\tau}) g(\boldsymbol{x}-\boldsymbol{\tau}) d\boldsymbol{\tau},
		\end{eqnarray}
		which satisfies
		\begin{eqnarray}
			\widehat{(f \star g)}(\boldsymbol{\omega}) =  (\hat{ f})(\boldsymbol{\omega}) (\hat{ g})(\boldsymbol{\omega}),
		\end{eqnarray}
		where $\hat{ f}$ is the multidimensional FT of $f$.
	\end{definition}	
	
	In the following two subsections, we define the multidimensional convolutions and correlation as extensions of those proposed by Prasad et al. in \cite{ta}, presented in Theorems \ref{con1} and \ref{con2}.
	
	\subsection{Convolution: type 1}
	\begin{theorem}\cite{ta} \label{con1}
		Let $f, g \in L^1(\mathbb{R})$. Assume $\tilde{F}(\omega) = (\mathcal{Q}_\Lambda f)(\omega)$ and $\tilde{G}(\omega)= (\mathcal{Q}_\Lambda g)(\omega)$, then
		\begin{eqnarray*}
			(\mathcal{Q}_\Lambda (f \otimes g))(\omega) = e^{-i(c \omega^2 + e\omega)} \tilde{F}(\omega) \tilde{G}(\omega),
		\end{eqnarray*}
		where
		\begin{eqnarray*}
			(f \otimes g)(x) = \sqrt{\frac{b}{i}} e^{-iax^2} \Big(\big(f(t) e^{i a x^2}\big) * \big(g(x)e^{iax^2}\big)\Big).
		\end{eqnarray*}
	\end{theorem} 
	Building on this, now we define the multidimensional extension of this convolution.
	
	\begin{definition}\label{def2}
		Let $f(\boldsymbol{x})$, $g(\boldsymbol{x}) \in L^1({\mathbb{R}^N})$. Then we define the convolution for multidimensional QPFT as follows
		\begin{eqnarray*}
			(f {\otimes} g)(\boldsymbol{x}) = \sqrt{\frac{\prod_{k=1}^{N} b_k}{(i)^N}} e^{-i(\sum_{k=1}^{N} a_k x_k^2)} \Big(\big(f(\boldsymbol{x})e^{i(\sum_{k=1}^{N} a_k x_k^2)}\big) \star \big(g(\boldsymbol{x})e^{i(\sum_{k=1}^{N} a_k x_k^2)} \big)\Big).
		\end{eqnarray*}
	\end{definition}
	
	\begin{theorem}
		Let the functions $f$ and $g$ be in $L^1({\mathbb{R}^N})$. Then 
		\begin{eqnarray*}
			(f {\otimes} g)(\boldsymbol{x}) \in L^1(\mathbb{R}^N) ~~~~and~~~~ ||f {\oplus} g||_1 \leq \sqrt{(2\pi )^N \prod_{k=1}^{N} b_k}~ ||f||_1 ||g||_1.
		\end{eqnarray*}
	\end{theorem}
	\begin{proof}
		Using the Definition \ref{def2} and Fubini's theorem, we get
		\begin{eqnarray*}
			\int_{\mathbb{R}^N}|(f {\otimes}~ g) (\boldsymbol{x})|d\boldsymbol{x} 
			&=& \small\sqrt{\prod_{k=1}^{N} b_k} \int_{\mathbb{R}^N}  \bigg{|} e^{-i(\sum_{k=1}^{N} a_k x_k^2)} \Big(\big(f(\boldsymbol{x})e^{i(\sum_{k=1}^{N} a_k x_k^2)}\big) \\&&\star \big(g(\boldsymbol{x})e^{i(\sum_{k=1}^{N} a_k x_k^2)} \big)\Big) \bigg{|} d\boldsymbol{x}
			\\&\leq& \sqrt{\prod_{k=1}^{N} b_k} \frac{1}{(\sqrt{2\pi})^N}\int_{\mathbb{R}^N}  \big{|}f(\boldsymbol{\tau}) \big{|} \bigg[ \int_{\mathbb{R}^N} \bigg{|} g(\boldsymbol{x}-\boldsymbol{\tau})\bigg{|}d\boldsymbol{x} \bigg] d\boldsymbol{\tau }
			\\&=& \sqrt{(2\pi )^N \prod_{k=1}^{N} b_k} ||f||_1 ||g||_1. 
		\end{eqnarray*}
		This completes the theorem.	
	\end{proof}
	Having established the definition and key properties, we are now in a position to state the convolution theorem for multidimensional QPFT.
	\begin{theorem}\label{conv1} (Convolution theorem for multidimensional QPFT) Let $\mathcal{Q}_{\boldsymbol{\Lambda}} f$ and $\mathcal{Q}_{\boldsymbol{\Lambda}} g$ be the multidimensional QPFTs of $f, g \in L^1 (\mathbb{R}^N)$ respectively. Then
		\begin{eqnarray*}
			\mathcal{Q}_{\boldsymbol{\Lambda}} (f \otimes g)(\boldsymbol{\omega}) 
			&=& e^{-i\sum_{k=1}^{N}(c_k \omega_k^2 + e_k \omega_k)}\mathcal{Q}_{\boldsymbol{\Lambda}}(g)(\boldsymbol{\omega}) \mathcal{Q}_{\boldsymbol{\Lambda}}(f)(\boldsymbol{\omega}).
		\end{eqnarray*}
	\end{theorem}
	\begin{proof} By using the Definitions \ref{def} and \ref{def2}, along with some calculations, we obtain
		\begin{eqnarray*}
			&& \mathcal{Q}_{\boldsymbol{\Lambda}} (f \otimes g)(\boldsymbol{\omega}) 
			\\&=&  \mathcal{Q}_{\boldsymbol{\Lambda}} \bigg(\sqrt{\frac{\prod_{k=1}^{N} b_k}{(2\pi i)^N}} e^{-i(\sum_{k=1}^{N} a_k x_k^2)} \\&& \times \Big( \int_{\mathbb{R}^N} f(\boldsymbol{\tau})e^{i(\sum_{k=1}^{N} a_k \tau_k^2)} g(\boldsymbol{x - \tau}) e^{i(\sum_{k=1}^{N} a_k (x_k- \tau_k)^2)}
			d\boldsymbol{\tau} \Big)\bigg)(\boldsymbol{\omega})
			\\&=&\sqrt{\frac{\prod_{k=1}^{N} b_k}{(2\pi i )^N}} e^{i \sum_{k=1}^{N}(c_k\omega_k^2 + e_k \omega_k)} \int_{\mathbb{R}^N}  f(\boldsymbol{\tau})e^{i(\sum_{k=1}^{N} a_k \tau_k^2)} \\&&\times \bigg(\sqrt{\frac{\prod_{k=1}^{N} b_k}{(2\pi i)^N}} \Big( \int_{\mathbb{R}^N} e^{i \sum_{k=1}^{N} (b_k x_k \omega_k +  d_k x_k)}  g(\boldsymbol{x - \tau})e^{i(\sum_{k=1}^{N} a_k (x_k- \tau_k)^2)}
			\Big) d\boldsymbol{x} \bigg) d\boldsymbol{\tau}.
		\end{eqnarray*}
		Setting $\boldsymbol{x - \tau} = \boldsymbol{\lambda}$, we obtain the following expression.
		\begin{eqnarray*}
			&& \mathcal{Q}_{\boldsymbol{\Lambda}} (f \otimes g)(\boldsymbol{\omega}) 
			\\&=& \sqrt{\frac{\prod_{k=1}^{N} b_k}{(2\pi i )^N}}  \int_{\mathbb{R}^N}  f(\boldsymbol{\tau})e^{i(\sum_{k=1}^{N} a_k \tau_k^2 + b_k \tau_k \omega_k+ d_k \tau_k)} \\&&\times \bigg(\sqrt{\frac{\prod_{k=1}^{N} b_k}{(2\pi i)^N}} \Big( \int_{\mathbb{R}^N} e^{i \sum_{k=1}^{N} (a_k \lambda_k^2+ b_k \lambda_k \omega_k + c_k\omega_k^2 +  d_k \lambda_k + e_k \omega_k)}  g(\boldsymbol{\lambda})
			\Big) d\boldsymbol{\lambda} \bigg)d\boldsymbol{\tau}
			\\&=& \sqrt{\frac{\prod_{k=1}^{N} b_k}{(2\pi i )^N}}  \int_{\mathbb{R}^N}  f(\boldsymbol{\tau})e^{i(\sum_{k=1}^{N} a_k \tau_k^2 + b_k \tau_k \omega_k+ d_k \tau_k)} \mathcal{Q}_{\boldsymbol{\Lambda}}(g)(\boldsymbol{\omega})d\boldsymbol{\tau}
			\\&=& e^{-i\sum_{k=1}^{N}(c_k \omega_k^2 + e_k \omega_k)}\mathcal{Q}_{\boldsymbol{\Lambda}}(g)(\boldsymbol{\omega}) \mathcal{Q}_{\boldsymbol{\Lambda}}(f)(\boldsymbol{\omega}).
		\end{eqnarray*}
		Hence, the proof is complete.
	\end{proof}
	
	\subsection{Convolution: type 2}
	\begin{definition}\cite{ta}(New convolution and correlation for QPFT) \label{con2}
		The new convolution operation $\oplus$ of QPFT for two signals $f,g \in L^1(\mathbb{R})$ is defined by
		\begin{eqnarray*}
			(f \oplus g) (x) = \sqrt{\frac{b}{\pi i}} \int_{\mathbb{R}} f(\tau) g(\sqrt{2}x - \tau) e^{2ai(\frac{x}{\sqrt{2}}- \tau)^2 + i (\sqrt{2}- 1)dx} d\tau,
		\end{eqnarray*}
		and a dual operation $\ominus$ of the new convolution operation is defined as
		\begin{eqnarray*}
			(f \ominus g) (x) = \sqrt{\frac{b i}{\pi}} \int_{\mathbb{R}} f(\tau) g(\sqrt{2}x - \tau) e^{2ci(\frac{x}{\sqrt{2}}- \tau)^2 - i (\sqrt{2}- 1)ex} d\tau.
		\end{eqnarray*}
	\end{definition}
	
	\begin{theorem} \cite{ta} Let $\mathcal{Q}_{\Lambda} f$ and $\mathcal{Q}_{\Lambda} g$ be the multidimensional QPFTs of $f, g \in L^1 (\mathbb{R})$ respectively. Then
		\begin{eqnarray*}
			\mathcal{Q}_\Lambda(f \oplus g) (\omega) = e^{-i(\sqrt{2}-1)e\omega} \mathcal{Q}_\Lambda(f)(\frac{\omega}{2}) \mathcal{Q}_\Lambda(g)(\frac{\omega}{2}), 
		\end{eqnarray*}
		and
		\begin{eqnarray*}
			\mathcal{Q}_\Lambda\Big[e^{i(\sqrt{2}-1)dx}f\big(\frac{x}{\sqrt{2}}\big) g\big(\frac{x}{\sqrt{2}}\big)\Big] (\omega) = \big( \mathcal{Q}_\Lambda(f) \ominus \mathcal{Q}_\Lambda(g)\big)(\omega).
		\end{eqnarray*}
	\end{theorem}
	We now define analogous convolution and a convolution
	theorem for functions on $\mathbb{R}^N$ to the above given theorem.
	\begin{definition}(New convolution and correlation for multidimensional QPFT)\label{conv3}
		The new convolution operation $\oplus$ of QPFT for two signals $f,g \in L^1(\mathbb{R}^N)$ is defined by
		\begin{eqnarray*}
			(f \boldsymbol{\oplus} g) (\boldsymbol{x}) = \sqrt{\frac{\prod_{k=1}^{N}b}{(\pi i)^N}} \int_{\mathbb{R}^N} f(\boldsymbol{\tau}) g(\sqrt{2}\boldsymbol{x} - \boldsymbol{\tau}) e^{\sum_{k=1}^{N}2ia_k(\frac{x_k}{\sqrt{2}}- \tau_k)^2 + i (\sqrt{2}- 1)d_kx_k} d\boldsymbol{\tau},
		\end{eqnarray*}
		and a dual operation $\boldsymbol{\ominus}$ of the new convolution operation is defined as
		\begin{eqnarray*}
			(f \boldsymbol{\ominus} g) (\boldsymbol{x}) = \sqrt{\frac{i^N\prod_{k=1}^{N}b_k }{(\pi)^N}} \int_{\mathbb{R}} f(\boldsymbol{\tau}) g(\sqrt{2}\boldsymbol{x} - \boldsymbol{\tau}) e^{-(\sum_{k=1}^{N}2c_ki(\frac{x_k}{\sqrt{2}}- \tau_k)^2 + i (\sqrt{2}- 1)e_kx_k)} d\boldsymbol{\tau_k}.
		\end{eqnarray*}
	\end{definition}	
	
	\begin{theorem} Let $\mathcal{Q}_{\boldsymbol{\Lambda}} f$ and $\mathcal{Q}_{\boldsymbol{\Lambda}} g$ be the multidimensional QPFTs of $f, g \in L^1 (\mathbb{R}^N)$ respectively. Then
		\begin{eqnarray*}
			\mathcal{Q}_{\boldsymbol{\Lambda}}(f \boldsymbol{\oplus}g) (\boldsymbol{\omega}) = e^{-i\sum_{k=1}^{N}(\sqrt{2}-1)e_k\omega_k} \mathcal{Q}_{\boldsymbol{\Lambda}}(f)(\frac{\boldsymbol{\omega}}{2}) \mathcal{Q}_{\boldsymbol{\Lambda}}(g)(\frac{\boldsymbol{\omega}}{2}), 
		\end{eqnarray*}
		and
		\begin{eqnarray*}
			\mathcal{Q}_{\boldsymbol{\Lambda}}\Big[e^{i\sum_{k=1}^{N}(\sqrt{2}-1)d_kx_k}f\big(\frac{\boldsymbol{x}}{\sqrt{2}}\big) g\big(\frac{\boldsymbol{x}}{\sqrt{2}}\big)\Big] (\boldsymbol{\omega}) = \big( \mathcal{Q}_{\boldsymbol{\Lambda}}(f) \boldsymbol{\ominus} \mathcal{Q}_{\boldsymbol{\Lambda}}(g)\big)(\boldsymbol{\omega}).
		\end{eqnarray*}
	\end{theorem}
	\begin{proof} By invoking Definitions \ref{def} and \ref{conv3} and carrying out the necessary computations, we obtain
		\begin{eqnarray*}
			&&\big( \mathcal{Q}_{\boldsymbol{\Lambda}}(f \boldsymbol{\oplus} g)(\boldsymbol{x})\big) (\boldsymbol{\omega}) 
			\\&=& \int_{\mathbb{R}^N} \mathcal{K}_{\boldsymbol{\Lambda}}(\boldsymbol{\omega},\boldsymbol{x}) \sqrt{\frac{\prod_{k=1}^{N}b_k}{(\pi i)^N}} \int_{\mathbb{R}^N} f(\boldsymbol{\tau}) g(\sqrt{2}\boldsymbol{x} - \boldsymbol{\tau}) e^{\sum_{k=1}^{N}2ia_k(\frac{x_k}{\sqrt{2}}- \tau_k)^2 + i (\sqrt{2}- 1)d_kx_k} d\boldsymbol{\tau} d\boldsymbol{x}
			\\&=& \sqrt{\frac{\prod_{k=1}^{N}b_k}{(\pi i)^N}} \sqrt{\frac{\prod_{k=1}^{N} b_k}{(2\pi i )^N}} \int_{\mathbb{R}^N} \int_{\mathbb{R}^N} e^{i \sum_{k=1}^{N}(a_k x_k^2 + b_k x_k \omega_k + c_k\omega_k^2 + d_k x_k + e_k \omega_k)} \\&& f(\boldsymbol{\tau}) g(\sqrt{2}\boldsymbol{x} - \boldsymbol{\tau}) e^{\sum_{k=1}^{N}2ia_k(\frac{x_k}{\sqrt{2}}- \tau_k)^2 + i (\sqrt{2}- 1)d_kx_k} d\boldsymbol{\tau} d\boldsymbol{x}.
		\end{eqnarray*} 
		Changing the variable $\sqrt{2}\boldsymbol{x} - \boldsymbol{\tau} = \boldsymbol{\epsilon}$, we get
		\begin{eqnarray*}
			&&\big( \mathcal{Q}_{\boldsymbol{\Lambda}}(f \boldsymbol{\oplus} g)(\boldsymbol{x})\big) (\boldsymbol{\omega}) 
			\\&=& \sqrt{\frac{\prod_{k=1}^{N}b_k}{(\pi i)^N}} \sqrt{\frac{\prod_{k=1}^{N} b_k}{(2\pi i )^N}} \int_{\mathbb{R}^N} \int_{\mathbb{R}^N} e^{i \sum_{k=1}^{N}(a_k \big(\frac{\tau_k^2}{2}+\frac{\epsilon_k^2}{2}+\tau_k\epsilon_k \big)+ b_k \big(\frac{\tau_k+\epsilon_k}{\sqrt{2}}\big) \omega_k + c_k\omega_k^2 + d_k \big(\frac{\tau_k+\epsilon_k}{\sqrt{2}}\big) + e_k \omega_k)} \\&& f(\boldsymbol{\tau}) g(\boldsymbol{\epsilon}) e^{\sum_{k=1}^{N}ia_k\big(\frac{\epsilon_k^2-2 \epsilon_k \tau_k + \tau_k^2}{2}\big) + i (\sqrt{2}- 1)d_k\big(\frac{\tau_k+\epsilon_k}{\sqrt{2}}\big)} d\boldsymbol{\tau} d\boldsymbol{\epsilon}\\
			&=& \small \sqrt{\frac{\prod_{k=1}^{N}b_k}{(\pi i)^N}} \sqrt{\frac{\prod_{k=1}^{N} b_k}{(2\pi i )^N}} e^{i\sum_{k=1}^{N}(1-\sqrt{2})e_k \omega_k} \int_{\mathbb{R}^N} e^{i \sum_{k=1}^{N}(a_k \tau_k^2 + b_k \big(\frac{\tau_k}{\sqrt{2}}\big) \omega_k+ c_k \big( \frac{\omega_k}{\sqrt{2}}\big)^2 + d_k\tau_k + e_k \frac{\omega_k}{\sqrt{2}})}f(\boldsymbol{\tau}) d\boldsymbol{\tau}\\&& \int_{\mathbb{R}^N}  e^{i \sum_{k=1}^{N} (a_k \epsilon_k^2 + b_k \big(\frac{\epsilon_k}{\sqrt{2}}\big) \omega_k  + c_k \big( \frac{\omega_k}{\sqrt{2}}\big)^2+ d_k\epsilon_k + e_k \frac{\omega_k}{\sqrt{2}})} g( \boldsymbol{\epsilon})  d\boldsymbol{\epsilon}\\
			&=& e^{i\sum_{k=1}^{N}(1-\sqrt{2})e_k \omega_k} \big( \mathcal{Q}_{\boldsymbol{\Lambda}}(f)\big) (\frac{\boldsymbol{\omega}}{\sqrt{2}}) \big( \mathcal{Q}_{\boldsymbol{\Lambda}}(g)\big) (\frac{\boldsymbol{\omega}}{\sqrt{2}}),
		\end{eqnarray*}
		which establishes the first part of the proof. We now proceed to the second part.
		\begin{eqnarray*}
			&& \big( \mathcal{Q}_{\boldsymbol{\Lambda}}(f) \boldsymbol{\ominus} \mathcal{Q}_{\boldsymbol{\Lambda}}(g)\big) (\boldsymbol{\omega}) 
			\\&=& \hspace{-0.1cm} \small \sqrt{\frac{i^N\prod_{k=1}^{N}b_k }{(\pi)^N}} \int_{\mathbb{R}^N} \mathcal{Q}_{\boldsymbol{\Lambda}}(f)(\boldsymbol{v}) \mathcal{Q}_{\boldsymbol{\Lambda}}(g)(\sqrt{2}\boldsymbol{\omega} - \boldsymbol{v})  e^{-(\sum_{k=1}^{N}2c_ki(\frac{\omega_k}{\sqrt{2}}- v_k)^2 + i (\sqrt{2}- 1)e_k\omega_k)} d\boldsymbol{v}\\
			&=& \sqrt{\frac{i^N\prod_{k=1}^{N}b_k }{(\pi)^N}} \sqrt{\frac{\prod_{k=1}^{N} b_k}{(2\pi i )^N}} \int_{\mathbb{R}^N} \int_{\mathbb{R}^N} e^{i \sum_{k=1}^{N}(a_k x_k^2 + b_k x_k v_k + c_k v_k^2 + d_k x_k + e_k v_k)}f(\boldsymbol{x})
			\\&& \times \mathcal{Q}_{\boldsymbol{\Lambda}}(g)(\sqrt{2}\boldsymbol{\omega} - \boldsymbol{v})  e^{-(\sum_{k=1}^{N}2c_ki(\frac{\omega_k}{\sqrt{2}}- v_k)^2 + i (\sqrt{2}- 1)e_k\omega_k)} d\boldsymbol{v} d\boldsymbol{x}.
		\end{eqnarray*}
		Changing the variable $\boldsymbol{\xi}= \sqrt{2} \boldsymbol{\omega}- \boldsymbol{v}$ to get
		\begin{eqnarray*}
			&&\big( \mathcal{Q}_{\boldsymbol{\Lambda}}(f) \boldsymbol{\ominus} \mathcal{Q}_{\boldsymbol{\Lambda}}(g)\big) (\boldsymbol{\omega}) 
			\\&=& \sqrt{\frac{\prod_{k=1}^{N} b_k}{(\pi i )^N}}  \int_{\mathbb{R}^N} \Big[ e^{i \sum_{k=1}^{N}(2 a_k x_k^2 + \sqrt{2}b_k x_k \omega_k+ c_k \omega_k^2 + 2 d_k x_k + e_k \omega_k)}f(\boldsymbol{x})
			\\&&\times \sqrt{\frac{i^N\prod_{k=1}^{N}b_k }{(2\pi)^N}} \int_{\mathbb{R}^N} \mathcal{Q}_{\boldsymbol{\Lambda}}(g)(\boldsymbol{\xi})  e^{-i \sum_{k=1}^{N} a_k x_k^2 + b_k \xi_k x_k + c_k \xi_k^2 +d_kx_k + e_k \omega_k} d\boldsymbol{\xi} \Big]  d\boldsymbol{x}\\
			&=& \sqrt{\frac{\prod_{k=1}^{N} b_k}{(\pi i )^N}}  \int_{\mathbb{R}^N} e^{i \sum_{k=1}^{N}(2 a_k x_k^2 + \sqrt{2}b_k x_k \omega_k+ c_k \omega_k^2 + 2 d_k x_k + e_k \omega_k)}f(\boldsymbol{x}) g(\boldsymbol{x}) d\boldsymbol{x}.
		\end{eqnarray*}
		Substituting $\boldsymbol{x} = \frac{\boldsymbol{t}}{\sqrt{2}}$ in the above equation gives
		\begin{eqnarray*}
			&&\big( \mathcal{Q}_{\boldsymbol{\Lambda}}(f) \ominus \mathcal{Q}_{\boldsymbol{\Lambda}}(g)\big) (\boldsymbol{\omega})
			\\ &=&  \hspace{-0.15cm}\small \sqrt{\frac{\prod_{k=1}^{N}b_k}{(2 \pi i)^N}} \int_{\mathbb{R}^N} e^{i\sum_{k=1}^{N}(a_kt_k^2+ b_k t_k \omega_k + c_k \omega_k^2 + d_k t_k +e_k \omega_k)} e^{i\sum_{k=1}^{N} (\sqrt{2}-1)d_k t_k} f\big(\frac{\boldsymbol{t}}{\sqrt{2}}\big) g\big(\frac{\boldsymbol{t}}{\sqrt{2}}\big) d\boldsymbol{t}
			\\&=& \mathcal{Q}_{\boldsymbol{\Lambda}} \big[ e^{i\sum_{k=1}^{N} (\sqrt{2}-1)d_k t_k} f\big(\frac{\boldsymbol{t}}{\sqrt{2}}\big) g\big(\frac{\boldsymbol{t}}{\sqrt{2}}\big)  \big].
		\end{eqnarray*}
		Thus, the proof is established.
	\end{proof}

	\subsection{Convolution: type 3}
	In this subsection, we introduce a more generalized version of convolution and establish corresponding convolution
	theorem in the multidimensional setting, which we use for proving the inversion
	theorem.
	\begin{definition}\label{def3}
		For $f \in L^1(\mathbb{R}^N) \cap L^2(\mathbb{R}^N)$ and $g \in L^1(\mathbb{R}^N)$, we define
		\begin{eqnarray*}
			(f {\otimes}_{\tiny \boldsymbol{\Lambda}, \lambda} ~g)(\boldsymbol{x}) &=& |\lambda|^\frac{N}{2} C_{\boldsymbol{\Lambda}}  e^{a, d}_{-\lambda^2}(\boldsymbol{x}) \big[ (e^{a, d}_{\lambda^2}f) * (e^{a, d}_{\lambda^2}g) \big] (\boldsymbol{x}), 
		\end{eqnarray*}
		where $e^{a, d}_{-\lambda^2}(\boldsymbol{x})$ and $C_{\boldsymbol{\Lambda}}$ are as in equations \eqref{ep} and \eqref{C}, respectively.
	\end{definition}
	Based on the given definition, we now formulate the convolution theorem.
	\begin{theorem}\label{conv2}
		Let $f, g \in L^1(\mathbb{R}^N)$ and $\mathcal{Q}_{\boldsymbol{\Lambda'}}f, \mathcal{Q}_{\boldsymbol{\Lambda'}}g \in L^1(\mathbb{R}^N)$, where $\boldsymbol{\Lambda'}=\{ \Lambda_1', \Lambda_2', . . . , \Lambda_N'\} \text{ and } \Lambda_k' = (a_k', b_k', c_k', d_k', e_k')$ with $a_k'= \lambda^2 a_k$,
		$b_k'= \lambda^2 b_k$, $c_k'= \lambda^2 c_k$, $d_k'= \lambda^2 d_k$, $e_k'= \lambda^2 e_k$. Then
		\begin{eqnarray*}
			\mathcal{Q}_{\boldsymbol{\Lambda'}} (f \otimes_{\tiny \boldsymbol{\Lambda}, \lambda} g)(\boldsymbol{\omega}) &=& e^{c, e}_{-\lambda^2}(\boldsymbol{\omega}) \mathcal{Q}_{\boldsymbol{\Lambda'}}(g)(\boldsymbol{\omega})\mathcal{Q}_{\boldsymbol{\Lambda'}}(f)(\boldsymbol{\omega}).
		\end{eqnarray*}
		
	\end{theorem}
	\begin{proof} 
		From Definition \ref{def3} and equations \eqref{ep} and \eqref{C}, it follows that 
		\begin{eqnarray*}
			&& \mathcal{Q}_{\boldsymbol{\Lambda'}} (f \otimes_{\tiny \boldsymbol{\Lambda}, \lambda} g)(\boldsymbol{\omega}) 
			\\&=&  \mathcal{Q}_{\boldsymbol{\Lambda'}} \bigg(\sqrt{\frac{\prod_{k=1}^{N} b_k'}{(i)^N}} e^{-i (\sum_{k=1}^{N} a_k' x_k^2+d_k' x_k)} \Big(\big(f(\boldsymbol{x})e^{i (\sum_{k=1}^{N} a_k' x_k^2+ d_k' x_k)}\big) \star \big(g(\boldsymbol{x})e^{i(\sum_{k=1}^{N} a_k' x_k^2 + d_k' x_k)} \big)\Big)\bigg)(\boldsymbol{\omega})
			\\&=&\sqrt{\frac{\prod_{k=1}^{N} b_k'}{(2\pi i )^N}} \int_{\mathbb{R}^N} e^{i \sum_{k=1}^{N}(a_k' x_k^2 + b_k' x_k \omega_k + c_k' \omega_k^2 + d_k' x_k + e_k' \omega_k)} \bigg(\sqrt{\frac{\prod_{k=1}^{N} b_k'}{(2\pi i)^N}} e^{-i  (\sum_{k=1}^{N} a_k' x_k^2 + d_k' x_k)} \\&&\hspace{-0.3cm} \times  \Big( \int_{\mathbb{R}^N} f(\boldsymbol{\tau})e^{i (\sum_{k=1}^{N} a_k' \tau_k^2 + d_k' \tau_k)} g(\boldsymbol{x - \tau})e^{i (\sum_{k=1}^{N} a_k' (x_k- \tau_k)^2 + d_k' (x_k - \tau_k))}
			d\boldsymbol{\tau} \Big)\bigg) d\boldsymbol{x}.
		\end{eqnarray*}
		Applying Fubini’s theorem, we can further simplify,
		\begin{eqnarray*}
			&& \mathcal{Q}_{\boldsymbol{\Lambda'}} \big((f \otimes_{\tiny \boldsymbol{\Lambda}, \lambda} g)(\boldsymbol{x})\big)(\boldsymbol{\omega}) \\&=&\sqrt{\frac{\prod_{k=1}^{N} b_k'}{(2\pi i )^N}} e^{i \sum_{k=1}^{N}(c_k'\omega_k^2 + e_k' \omega_k)} \int_{\mathbb{R}^N}  f(\boldsymbol{\tau})e^{i(\sum_{k=1}^{N} a_k' \tau_k^2 + d_k' \tau_k)} \\&&\times \bigg(\sqrt{\frac{\prod_{k=1}^{N} b_k'}{(2\pi i)^N}} \Big( \int_{\mathbb{R}^N} e^{i\sum_{k=1}^{N} (b_k' x_k \omega_k)}  g(\boldsymbol{x - \tau})e^{i (\sum_{k=1}^{N} a_k' (x_k- \tau_k)^2 + d_k'(x_k - \tau_k))}
			\Big) d\boldsymbol{x} \bigg) d\boldsymbol{\tau}.
		\end{eqnarray*}
		With the substitution $\boldsymbol{x} - \boldsymbol{\tau} = \boldsymbol{u}$, we get
		\begin{eqnarray*}
			&& \mathcal{Q}_{\boldsymbol{\Lambda}'} (f \otimes g)(\boldsymbol{\omega})
			\Big) d\boldsymbol{u} \bigg) d\boldsymbol{\tau}
			\\&=& \sqrt{\frac{\prod_{k=1}^{N} b_k'}{(2\pi i )^N}}  \int_{\mathbb{R}^N}  f(\boldsymbol{\tau})e^{i(\sum_{k=1}^{N} a_k' \tau_k^2 + b_k' \tau_k \omega_k+ d_k' \tau_k)} \\&&\times \bigg(\sqrt{\frac{\prod_{k=1}^{N} b_k'}{(2\pi i)^N}} \Big( \int_{\mathbb{R}^N} e^{i \sum_{k=1}^{N} (a_k' u_k^2+ b_k' u_k \omega_k + c_k'\omega_k^2 +  d_k' u_k + e_k' \omega_k)}  g(\boldsymbol{u})
			\Big) d\boldsymbol{u} \bigg)d\boldsymbol{\tau}
			\\&=& \sqrt{\frac{\prod_{k=1}^{N} b_k'}{(2\pi i )^N}}  \int_{\mathbb{R}^N}  f(\boldsymbol{\tau})e^{i(\sum_{k=1}^{N} a_k' \tau_k^2 + b_k' \tau_k \omega_k+ d_k' \tau_k)} \mathcal{Q}_{\boldsymbol{\Lambda'}}(g)(\boldsymbol{\omega})d\boldsymbol{\tau}
			\\&=& e^{-i\sum_{k=1}^{N}(c_k' \omega_k^2 + e_k' \omega_k)}\mathcal{Q}_{\boldsymbol{\Lambda'}}(g)(\boldsymbol{\omega}) \mathcal{Q}_{\boldsymbol{\Lambda'}}(f)(\boldsymbol{\omega})
			\\&=& e_{-\lambda^2}^{c,e}(\boldsymbol{\omega}) \mathcal{Q}_{\boldsymbol{\Lambda'}}(g)(\boldsymbol{\omega}) \mathcal{Q}_{\boldsymbol{\Lambda'}}(f)(\boldsymbol{\omega}).
		\end{eqnarray*}
		This concludes the proof.
	\end{proof}
	\begin{definition} \label{def4}
		Let $f \in L^1(\mathbb{R}^N) \cap L^2(\mathbb{R}^N)$, $\boldsymbol{t} \in \mathbb{R}^N$ and $\boldsymbol{\omega} \in (\mathbb{R} \setminus \{0\})^N$. Define $(T_{\boldsymbol{t}} f)(\boldsymbol{x}) = f(\boldsymbol{x} - \boldsymbol{t})$ and $
		E_{\lambda^2}^{a \boldsymbol{t}} (\boldsymbol{x}) =e^{2i \lambda^2 (\sum_{k=1}^{N} a_k x_k t_k)}, E_{\lambda^2}^{d} (\boldsymbol{t})= e^{2i \lambda^2 (\sum_{k=1}^{N} d_k t_k)} .$
	\end{definition}
	With these notations in place, we now state some properties of the previously defined convolution \ref{def3}.
	\begin{theorem}
		Let the functions $f$, $f_1$ and $f_2$ be in $L^p(\mathbb{R}^N)$ and $g$, $g_1$ and $g_2$ be in $L^1(\mathbb{R}^N)$ and $\gamma$ be a complex number, $\omega$ be a non-zero real number and $t \in \mathbb{R}^N$. Then
		\begin{enumerate}
			\item $(f_1 + f_2) \otimes_{\boldsymbol{\Lambda}, \lambda} h = (f_1 \otimes_{\boldsymbol{\Lambda}, \lambda} h) + (f_2 \otimes_{\boldsymbol{\Lambda}, \lambda} h)$; distributivity,
			\vspace{0.15cm}
			\item $\gamma (f \otimes_{\boldsymbol{\Lambda}, \lambda} g) = (\gamma f) \otimes_{\boldsymbol{\Lambda},\lambda} g = f \otimes_{\boldsymbol{\Lambda},\lambda} (\gamma g)$; scalar multiplication,
			\vspace{0.15cm}
			\item $ [f \otimes_{\boldsymbol{\Lambda}, \lambda} (g_1 \otimes_{\boldsymbol{\Lambda}, \lambda} g_2)] = [(f \otimes_{\boldsymbol{\Lambda}, \lambda} g_1) \otimes_{\boldsymbol{\Lambda}, \lambda} g_2]$; associativity,
			\vspace{0.15cm}
			\item $f \otimes_{\boldsymbol{\Lambda}, \lambda} g = g \otimes_{\boldsymbol{\Lambda}, \lambda} f$; commutativity ,
			\vspace{0.15cm}
			\item $\tau_{\boldsymbol{t}} (f \otimes_{\boldsymbol{\Lambda}, \lambda} g) (\boldsymbol{x}) =\big[(\tau_{\boldsymbol{t}} f) \otimes_{\boldsymbol{\Lambda}, \lambda} (E_{\lambda^2}^{a\boldsymbol{t}} g)\big] (\boldsymbol{x}).$
			\vspace{0.15cm}
			\item Define a function $P_k: \mathbb{R}^N \to \mathbb{R}$ by $P_k(\boldsymbol{x})= x_k$, $k= 1,2, . . . , N.$
			\vspace{0.15cm}
			\begin{enumerate}
				\item If $P_kf, \partial_{x_k}f  \in L^1(\mathbb{R}^N)$ or $L^2(\mathbb{R}^N)$, 
				then
				\begin{eqnarray*}
					\partial_{x_k}(f \otimes_{\boldsymbol{\Lambda}, \lambda} g) (\boldsymbol{x}) &=&  2a_k i\lambda^2 \big[( P_k f \otimes_{\boldsymbol{\Lambda}, \lambda} g)(\boldsymbol{x}) - P_k(f \otimes_{\boldsymbol{\Lambda}, \lambda} g)(\boldsymbol{x})\big] 
					\\&& + ( \partial_{x_k}f \otimes_{\boldsymbol{\Lambda}, \lambda} g)(\boldsymbol{x}).
				\end{eqnarray*}
				\item If $P_k g$ and $\partial_{x_k} g$ are in $L^1(\mathbb{R}^N)$, then
				\begin{eqnarray*}
					\partial_{x_k}(f \otimes_{\boldsymbol{\Lambda}, \lambda} g) (\boldsymbol{x}) &=&  2a_k i\lambda^2 \big[( P_k f \otimes_{\boldsymbol{\Lambda}, \lambda} g)(\boldsymbol{x}) - P_k(f \otimes_{\boldsymbol{\Lambda}, \lambda} g)(\boldsymbol{x})\big] 
					\\&& +  ( \partial_{x_k}f \otimes_{\boldsymbol{\Lambda}, \lambda} g)(\boldsymbol{x}).
				\end{eqnarray*}
			\end{enumerate}
		\end{enumerate}
	\end{theorem}
	\begin{proof}
		The proofs of statements (1) to (4) follow directly from the corresponding results in the classical convolution framework, requiring only minor modifications for the generalized structure. Therefore, we omit the detailed proofs.
 		\newline
 		\newline
		(5) To begin, we observe that
		\begin{eqnarray} \label{tau}
			\tau_{\boldsymbol{t}}(e^{\boldsymbol{a}, \boldsymbol{d}}_{-\lambda^2}(\boldsymbol{x})) 
			&=& e^{-i \lambda^2 (\sum_{k=1}^{N} a_k (x_k- t_k)^2+d_k (x_k - t_k))} \nonumber\\
			&=& e^{-i \lambda^2 (\sum_{k=1}^{N} a_k x_k^2 +d_k x_k + a_k t_k^2 +d_k t_k  - 2(a_k x_k t_k + d_k t_k )} \nonumber\\
			&=& e^{-i \lambda^2 (\sum_{k=1}^{N} a_k x_k^2 + d_k x_k)} e^{-i \lambda^2 (\sum_{k=1}^{N} a_k t_k^2 + d_k t_k)} e^{2i \lambda^2 (\sum_{k=1}^{N} a_k x_k t_k)} e^{2i \lambda^2 (\sum_{k=1}^{N} d_k t_k)} \nonumber \\
			&=& e^{\boldsymbol{a}, \boldsymbol{d}}_{-\lambda^2}(\boldsymbol{x}) e^{\boldsymbol{a}, \boldsymbol{d}}_{-\lambda^2}(\boldsymbol{t}) E_{\lambda^2}^{\boldsymbol{at}} (\boldsymbol{x})  E_{\lambda^2}^{\boldsymbol{d}} (\boldsymbol{t}).
		\end{eqnarray}
		Thus by using the identity $\tau_{\boldsymbol{t}}(f * g) = (\tau_{\boldsymbol{t}} f) * g$, we obtain
		\begin{eqnarray*}
			\tau_{\boldsymbol{t}} ( f \otimes_{\boldsymbol{\Lambda}, \lambda} g)(\boldsymbol{x})
			&=& |\lambda|^\frac{N}{2} C_{\boldsymbol{\Lambda}}  \tau_{\boldsymbol{t}}(e^{\boldsymbol{a}, \boldsymbol{d}}_{-\lambda^2}(\boldsymbol{x})) \big( \tau_{\boldsymbol{t}}(e^{\boldsymbol{a}, \boldsymbol{d}}_{\lambda^2}f) * (e^{\boldsymbol{a}, \boldsymbol{d}}_{\lambda^2}g) \big) (\boldsymbol{x}).
		\end{eqnarray*}
		Now applying \eqref{tau} to the preceding equation yields
		\begin{eqnarray*}
			\tau_{\boldsymbol{t}} ( f \otimes_{\boldsymbol{\Lambda}, \lambda} g)(\boldsymbol{x}) 
			&=& |\lambda|^\frac{N}{2} C_{\boldsymbol{\Lambda}}   e^{\boldsymbol{a}, \boldsymbol{d}}_{-\lambda^2}(\boldsymbol{x}) e^{\boldsymbol{a}, \boldsymbol{d}}_{-\lambda^2}(\boldsymbol{t}) E_{\lambda^2}^{\boldsymbol{a}\boldsymbol{t}} (\boldsymbol{x})  E_{\lambda^2}^{\boldsymbol{d}} (\boldsymbol{t}) \\&& \times \Big(  e^{\boldsymbol{a}, \boldsymbol{d}}_{\lambda^2}(\boldsymbol{t}) E_{-\lambda^2}^{\boldsymbol{d}} (\boldsymbol{t}) e^{\boldsymbol{a}, \boldsymbol{d}}_{\lambda^2} E_{-\lambda^2}^{\boldsymbol{a}\boldsymbol{t}} \tau_{\boldsymbol{t}}(f) * (e^{\boldsymbol{a}, \boldsymbol{d}}_{\lambda^2}g) \Big) (\boldsymbol{x})
			\\&=& |\lambda|^\frac{N}{2} C_{\boldsymbol{\Lambda}}   e^{\boldsymbol{a}, \boldsymbol{d}}_{-\lambda^2}(\boldsymbol{x}) E_{\lambda^2}^{\boldsymbol{a} \boldsymbol{t}} (\boldsymbol{x}) \Big(e^{\boldsymbol{a}, \boldsymbol{d}}_{\lambda^2} E_{-\lambda^2}^{\boldsymbol{a}\boldsymbol{t}} \tau_{\boldsymbol{t}}(f) * (e^{\boldsymbol{a}, \boldsymbol{d}}_{\lambda^2}g) \Big) (\boldsymbol{x})
			\\&=& |\lambda|^\frac{N}{2} C_{\boldsymbol{\Lambda}}   e^{\boldsymbol{a}, \boldsymbol{d}}_{-\lambda^2}(\boldsymbol{x})  \Big(e^{\boldsymbol{a}, \boldsymbol{d}}_{\lambda^2} \tau_{\boldsymbol{t}}(f) * (e^{\boldsymbol{a}, \boldsymbol{d}}_{\lambda^2}  E_{\lambda^2}^{\boldsymbol{a}\boldsymbol{t}} g) \Big) (\boldsymbol{x})
			\\&=& \big[(\tau_{\boldsymbol{t}} f) \otimes_{\boldsymbol{\Lambda}, \lambda} (E_{\lambda^2}^{\boldsymbol{a}\boldsymbol{t}} g)\big] (\boldsymbol{x}).
		\end{eqnarray*}
		Having established this result, we now proceed to the next part.
		\newline
		\newline
		(6) (a) Let $f \in L^p(\mathbb{R}^N)$ and $g \in L^1(\mathbb{R}^N)$. By the convolution \ref{def3}, we get 
		\begin{eqnarray*}
			&&\partial_{x_k}( f \otimes_{\boldsymbol{\Lambda}, \lambda} g)(\boldsymbol{x}) 
			\\&=& |\lambda|^\frac{N}{2} C_{\boldsymbol{\Lambda}}   \partial_{x_k} \big(e^{a, d}_{-\lambda^2}(\boldsymbol{x})\big) \big( (e^{a, d}_{\lambda^2}f) * (e^{a, d}_{\lambda^2}g) \big)(\boldsymbol{x}) +  |\lambda|^\frac{N}{2} C_{\boldsymbol{\Lambda}}  e^{a, d}_{-\lambda^2}(\boldsymbol{x}) \partial_{x_k} \big( (e^{a, d}_{\lambda^2}f) * (e^{a, d}_{\lambda^2}g) \big)(\boldsymbol{x})
		\end{eqnarray*}
		Using the differentiation property of convolution,  
		\begin{eqnarray*}
			\partial_{x_k} (f * g)(\boldsymbol{x}) = (\partial_{x_k} f * g)(\boldsymbol{x}),
		\end{eqnarray*}it follows that
		\begin{eqnarray*}	
			&&\partial_{x_k}( f \otimes_{\boldsymbol{\Lambda}, \lambda} g)(\boldsymbol{x})\\&=& |\lambda|^\frac{N}{2} C_{\boldsymbol{\Lambda}}   \partial_{x_k} \big(e^{a, d}_{-\lambda^2}(\boldsymbol{x})\big) \big( (e^{a, d}_{\lambda^2}f) * (e^{a, d}_{\lambda^2}g) \big)(\boldsymbol{x}) +  |\lambda|^\frac{N}{2} C_{\boldsymbol{\Lambda}}  e^{a, d}_{-\lambda^2}(\boldsymbol{x})  \big( \partial_{x_k}(e^{a, d}_{\lambda^2}f) * (e^{a, d}_{\lambda^2}g) \big)(\boldsymbol{x})
			\\&=& |\lambda|^\frac{N}{2} C_{\boldsymbol{\Lambda}}   (-i\lambda^2)(2a_kx_k+ d_k) \big(e^{a, d}_{-\lambda^2}(\boldsymbol{x})\big) \big( (e^{a, d}_{\lambda^2}f) * (e^{a, d}_{\lambda^2}g) \big)(\boldsymbol{x}) 
			\\&&+  |\lambda|^\frac{N}{2} C_{\boldsymbol{\Lambda}}  e^{a, d}_{-\lambda^2}(\boldsymbol{x})  \big[ (i\lambda^2)(2a_kx_k+ d_k) \big(e^{a, d}_{-\lambda^2}(\boldsymbol{x})\big) f + \big(e^{a, d}_{-\lambda^2}(\boldsymbol{x})\big) (\partial_{x_k}f) \big]  * (e^{a, d}_{\lambda^2}g) \big)(\boldsymbol{x})
			\\&=&  |\lambda|^\frac{N}{2} C_{\boldsymbol{\Lambda}}   (i\lambda^2) \Big[e^{a, d}_{-\lambda^2}(\boldsymbol{x})  \big[ \big(e^{a, d}_{\lambda^2}(\boldsymbol{x})\big) (2a_kx_k+ d_k) f  * (e^{a, d}_{\lambda^2}g) \big)(\boldsymbol{x}) 
			\\&& - (2a_kx_k+ d_k) \big(e^{a, d}_{-\lambda^2}(\boldsymbol{x})\big) \big( (e^{a, d}_{\lambda^2}f) * (e^{a, d}_{\lambda^2}g) \big)(\boldsymbol{x}) \Big]
			\\&& +  |\lambda|^\frac{N}{2} C_{\boldsymbol{\Lambda}}  e^{a, d}_{-\lambda^2}(\boldsymbol{x}) \big(e^{a, d}_{\lambda^2}(\boldsymbol{x})\big) (\partial_{x_k}f)  * (e^{a, d}_{\lambda^2}g) \big)(\boldsymbol{x})
			\\&=&  2a_k i |\lambda|^{\frac{N}{2}+2} C_{\boldsymbol{\Lambda}} \Big[e^{a, d}_{-\lambda^2}(\boldsymbol{x})  \big[ \big(e^{a, d}_{\lambda^2}(\boldsymbol{x})\big) x_k f  * (e^{a, d}_{\lambda^2}g) \big)(\boldsymbol{x}) - x_k \big(e^{a, d}_{-\lambda^2}(\boldsymbol{x})\big) \big( (e^{a, d}_{\lambda^2}f) * (e^{a, d}_{\lambda^2}g) \big)(\boldsymbol{x}) \Big]
			\\&&+ i d_k |\lambda|^{\frac{N}{2}+2} C_{\boldsymbol{\Lambda}} \Big[e^{a, d}_{-\lambda^2}(\boldsymbol{x})  \big[ \big(e^{a, d}_{\lambda^2}\big) f  * (e^{a, d}_{\lambda^2}g) \big](\boldsymbol{x}) - \big(e^{a, d}_{-\lambda^2}(\boldsymbol{x})\big) \big( (e^{a, d}_{\lambda^2}f) * (e^{a, d}_{\lambda^2}g) \big)(\boldsymbol{x}) \Big]
			\\&& +  |\lambda|^\frac{N}{2} C_{\boldsymbol{\Lambda}}  e^{a, d}_{-\lambda^2}(\boldsymbol{x}) \big(e^{a, d}_{\lambda^2}(\boldsymbol{x})\big) \partial_{x_k}f  * (e^{a, d}_{\lambda^2}g) \big)(\boldsymbol{x})
			\\&=& 2a_k i |\lambda|^{\frac{N}{2}+2} C_{\boldsymbol{\Lambda}} \Big[e^{a, d}_{-\lambda^2}(\boldsymbol{x})  \big[ \big(e^{a, d}_{\lambda^2}(\boldsymbol{x})\big) P_k(\boldsymbol{x}) f  * (e^{a, d}_{\lambda^2}g \big)(\boldsymbol{x}) \big] \\&&- P_k(\boldsymbol{x}) \big(e^{a, d}_{-\lambda^2}(\boldsymbol{x})\big) \big( (e^{a, d}_{\lambda^2}f) * (e^{a, d}_{\lambda^2}g) \big)(\boldsymbol{x}) \Big]
			+  ( \partial_{x_k}f \otimes_{\boldsymbol{\Lambda}, \lambda} g)(\boldsymbol{x}) 
			\\&=&  2a_k i\lambda^2 \big[( P_k f \otimes_{\boldsymbol{\Lambda}, \lambda} g)(\boldsymbol{x}) - P_k(f \otimes_{\boldsymbol{\Lambda}, \lambda} g)(\boldsymbol{x})\big] +  ( \partial_{x_k}f \otimes_{\boldsymbol{\Lambda}, \lambda} g)(\boldsymbol{x}). 
		\end{eqnarray*} 
		(b) The proof follows similarly to part (a). 
		\newline
		\newline
		Hence, the proof is complete.
	\end{proof}
	
	\section{Inversion theorem}
	In this section, we establish the inversion theorem for the multidimensional QPFT using the convolution theorem \ref{def3} and its properties. Before moving ahead, we state two lemmas that supports our main results.
	\begin{lemma}\label{lem1}
		Let $\boldsymbol{\Lambda} = (\Lambda_1, \Lambda_2, . . . , \Lambda_N)$	such that $\Lambda_k = (a_k, b_k, . . . ,e_k)$, $b_k \neq 0$, $k = 1, 2, . . . , N$. Moreover $\lambda > 0$, 
		 \begin{eqnarray*}
		 	H_\lambda (\boldsymbol{x}) = \prod_{k=1}^{N} e^{-|\lambda x_k|} \text{ and }		 
		 \end{eqnarray*} 
		\begin{eqnarray}\label{h}
			h_{\boldsymbol{\Lambda}, \lambda} (\boldsymbol{\omega}) = \int_{\mathbb{R}^N} H_\lambda(\boldsymbol{x}) e_1^{\boldsymbol{c, e}} (\boldsymbol{x}) \mathcal{K}_{-\boldsymbol{\Lambda}} (\boldsymbol{\omega}, \boldsymbol{x}) d\boldsymbol{x}.
		\end{eqnarray}
		Then the following holds
		\begin{enumerate}
			\item $H_\lambda (\boldsymbol{x}) = \prod_{k=1}^{N} e^{-|\lambda x_k|} \to 1 \text{ as } \lambda \to 0 \text{ for each } \boldsymbol{x} \in \mathbb{R}^N$,
			\vspace{0.15cm}
			\item $\frac{C_{\boldsymbol{\Lambda}}}{(2\pi)^\frac{N}{2}} \int_{\mathbb{R}^N} h_{\boldsymbol{\Lambda}, \lambda} (\boldsymbol{\omega}) e^{\boldsymbol{a, d}} (\boldsymbol{\omega}) d\boldsymbol{\omega} = 1$,
			\vspace{0.15cm}
			\item $h_{\boldsymbol{\Lambda}, \lambda} \in L^p (\mathbb{R}^N),~~ p \in [1, \infty), \text{ if } |C_{\boldsymbol{\Lambda}}| \leq 1$,
			\vspace{0.15cm}
			\item $h_{\boldsymbol{\Lambda}, \lambda}(\boldsymbol{\omega}) = \frac{e_1^{\boldsymbol{a, d}}(\boldsymbol{\omega}) e_1^{\boldsymbol{c, e}}(\frac{\boldsymbol{\omega}}{\lambda})}{\lambda^N} h_{\boldsymbol{\Lambda}, 1} (\frac{\boldsymbol{\omega}}{\lambda})$.
		\end{enumerate}
	\end{lemma}
	\begin{proof}
		(1) We observe that $e^{|\lambda x_k|}$ approaches $1$, as $\lambda$ tends to $0$, for all $k =1, 2, . . . , N$. So the desired result follows immediately.
		\\(2) By using the equations \eqref{h} and \eqref{ep}, we get 
		\begin{eqnarray}\label{eqn4.2}
			h_{\boldsymbol{\Lambda}, \lambda} (\boldsymbol{\omega})
			&=& \prod_{k=1}^{N} \int_{\mathbb{R}}  e^{-|\lambda x_k|} e^{i (c_k x_k^2 + e_k x_k)} \mathcal{K}_{-\Lambda_k} (x_k, \omega_k) d x_k \nonumber\\
			&=& \frac{C_{-\boldsymbol{\Lambda}}}{(2 \pi)^\frac{N}{2}} \prod_{k=1}^{N} \int_{\mathbb{R}}  e^{-|\lambda x_k|} e^{i (c_k x_k^2 + e_k x_k)} ~e^{-i(a_k \omega_k^2 + b_k x_k \omega_k + c_k x_k^2 + d_k \omega_k + e_k x_k)} d x_k \nonumber
			\\&=& \frac{C_{-\boldsymbol{\Lambda}}}{(2 \pi)^\frac{N}{2}} ~e^{-i(a_k \omega_k^2 + d_k \omega_k)} \prod_{k=1}^{N} \int_{\mathbb{R}}  e^{-|\lambda x_k|} ~e^{-ib_k x_k \omega_k} d x_k \nonumber
			\\&=& \frac{C_{-\boldsymbol{\Lambda}}}{(2 \pi)^\frac{N}{2}} ~e_{-1}^{\boldsymbol{a, d}}(\boldsymbol{\omega}) \prod_{k=1}^{N} \bigg[ \int_{-\infty}^0  e^{(\lambda - i b_k \omega_k) x_k} d x_k + \int_0^{\infty}  e^{-(\lambda + i b_k \omega_k) x_k} d x_k \bigg] \nonumber
			\\&=& \frac{C_{-\boldsymbol{\Lambda}}}{(2 \pi)^\frac{N}{2}} ~e_{-1}^{\boldsymbol{a, d}}(\boldsymbol{\omega}) \prod_{k=1}^{N} \biggl\{\bigg[   \frac{1}{(\lambda - i b_k \omega_k)} \bigg] + \bigg[ \frac{1}{(\lambda + i b_k \omega_k)} \bigg]\biggr\} \nonumber
			\\&=& \frac{C_{-\boldsymbol{\Lambda}}}{(2 \pi)^\frac{N}{2}} ~e_{-1}^{\boldsymbol{a, d}}(\boldsymbol{\omega}) \prod_{k=1}^{N} \biggl\{ \frac{2\lambda}{(\lambda^2 +  b_k^2 \omega_k^2)} \biggr\}.			
		\end{eqnarray}
		This implies
		\begin{eqnarray*}
			\frac{C_{\boldsymbol{\Lambda}}}{(2 \pi)^\frac{N}{2}} h_{\boldsymbol{\Lambda}, \lambda} (\boldsymbol{\omega}) ~e_{1}^{\boldsymbol{a, d}}(\boldsymbol{\omega}) 
			&=& \frac{C_{\boldsymbol{\Lambda}}}{(2 \pi)^\frac{N}{2}}  \frac{C_{-\boldsymbol{\Lambda}}}{(2 \pi)^\frac{N}{2}} \prod_{k=1}^{N} \biggl\{ \frac{2\lambda}{(\lambda^2 +  b_k^2 \omega_k^2)} \biggr\}\\
			&=& \frac{1} {(\pi)^N} \prod_{k=1}^{N} \biggl\{ \frac{\lambda b_k}{(\lambda^2 +  b_k^2 \omega_k^2)} \biggr\}.		
		\end{eqnarray*}
		Integrating on both sides of the equation to get
		\begin{eqnarray*}
			\frac{C_{\boldsymbol{\Lambda}}}{(2 \pi)^\frac{N}{2}} \int_{\mathbb{R}^N} h_{\boldsymbol{\Lambda}, \lambda} (\boldsymbol{\omega}) ~e_{1}^{\boldsymbol{a, d}}(\boldsymbol{\omega}) d\boldsymbol{\omega}
			&=& \frac{1} {(\pi)^N} \prod_{k=1}^{N} \int_{\mathbb{R}} \biggl\{ \frac{\lambda b_k}{(\lambda^2 +  b_k^2 \omega_k^2)} \biggr\} d\omega_k
			\\&=& \frac{1} {(\pi)^N} \prod_{k=1}^{N} \int_{\mathbb{R}} \biggl\{ \frac{\lambda }{(\lambda^2 +  z_k^2)} \biggr\} dz_k,		
		\end{eqnarray*}
		where $z_k = b_k \omega_k$. On further simplification, we get
		\begin{eqnarray*}
			\frac{C_{\boldsymbol{\Lambda}}}{(2 \pi)^\frac{N}{2}} \int_{\mathbb{R}^N} h_{\boldsymbol{\Lambda}, \lambda} (\boldsymbol{\omega}) ~e_{1}^{\boldsymbol{a, d}}(\boldsymbol{\omega}) d\boldsymbol{\omega}
			&=& \frac{1} {(\pi)^N} \prod_{k=1}^{N}  \biggl[ \arctan \Big(\frac{z_k}{\lambda}\Big) \biggr]_{-\infty}^{\infty}\\
			&=& \frac{1} {(\pi)^N} \prod_{k=1}^{N} \biggl\{ \frac{\pi}{2}-\frac{-\pi}{2} \biggr\}
			\\&=& \frac{1} {(\pi)^N} \prod_{k=1}^{N} \pi
			\\&=& 1.
		\end{eqnarray*}
		Hence the result.
		\\(3) For $p \in [1, \infty)$, using \eqref{eqn4.2}, we obtain
		\begin{eqnarray*}
			||h_{\boldsymbol{\Lambda}, \lambda} ||_p^p 
			&=& \frac{1}{(2 \pi)^{\frac{N}{2}}} \int_{\mathbb{R}^N} \bigg| \frac{C_{-\boldsymbol{\Lambda}}}{(2 \pi)^\frac{N}{2}} ~e_{-1}^{\boldsymbol{a, d}}(\boldsymbol{\omega}) \prod_{k=1}^{N} \biggl\{ \frac{2\lambda}{(\lambda^2 +  b_k^2 \omega_k^2)} \biggr\} \bigg|^p d\boldsymbol{\omega}\\
			&=& \frac{|C_{-\boldsymbol{\Lambda}}|^p}{\lambda^{pN}(2 \pi)^\frac{N}{2}} \int_{\mathbb{R}^N} \bigg| \bigg( \sqrt{\frac{2}{\pi}}\bigg)^N \prod_{k=1}^{N} \biggl\{ \frac{\lambda^2}{(\lambda^2 +  b_k^2 \omega_k^2)} \biggr\} \bigg|^p d\boldsymbol{\omega}\\
			&=& \frac{|C_{-\boldsymbol{\Lambda}}|^p}{\lambda^{pN}(2 \pi)^\frac{N}{2}} \prod_{k=1}^{N} \int_{\mathbb{R}}  \bigg( \sqrt{\frac{2}{\pi}} \bigg)^p  \biggl\{ \frac{\lambda^2}{(\lambda^2 +  b_k^2 \omega_k^2)} \biggr\}^p d\omega_k\\
			&\leq& \frac{|C_{-\boldsymbol{\Lambda}}|^p}{\lambda^{pN}(2 \pi)^\frac{N}{2}} \prod_{k=1}^{N} \int_{\mathbb{R}}  \bigg( \sqrt{\frac{2}{\pi}} \bigg)  \biggl\{ \frac{\lambda^2}{(\lambda^2 +  b_k^2 \omega_k^2)} \biggr\} d\omega_k\\
			&=& \frac{|C_{-\boldsymbol{\Lambda}}|^p}{\lambda^{pN-N}(\pi)^N} \bigg( \prod_{k=1}^{N} \frac{1}{b_k} \bigg)\prod_{k=1}^{N} \int_{\mathbb{R}}    \biggl\{ \frac{\lambda b_k}{(\lambda^2 +  b_k^2 \omega_k^2)} \biggr\} d\omega_k\\		
			&=& \frac{|C_{-\boldsymbol{\Lambda}}|^{p-2}}{\lambda^{pN-N}(\pi)^N}  \prod_{k=1}^{N} \int_{\mathbb{R}}    \biggl\{ \frac{\lambda b_k}{(\lambda^2 +  b_k^2 \omega_k^2)} \biggr\} d\omega_k\\
			&=& \frac{|C_{-\boldsymbol{\Lambda}}|^{p-2}}{\lambda^{pN-N}(\pi)^N} \prod_{k=1}^{N} \pi \\
			&=& \frac{|C_{-\boldsymbol{\Lambda}}|^{p-2}}{\lambda^{pN-N}}
			~~ < \infty.								
		\end{eqnarray*}
		Therefore
		$h_{\boldsymbol{\Lambda}, \lambda} \in L^p(\mathbb{R}^N)$.
		\\(4) Making the change of variable $\lambda x_k = y_k$ in the equation \eqref{h}, we get
		\begin{eqnarray*}
			h_{\boldsymbol{\Lambda}, \lambda}(\boldsymbol{\omega}) &=& \prod_{k=1}^{N} \int_{\mathbb{R}}  e^{-|y_k|} e^{i \big(c_k \big(\frac{y_k}{\lambda}\big)^2 + e_k (\frac{y_k}{\lambda})\big)} \mathcal{K}_{-\Lambda_k} (\frac{y_k}{\lambda}, \omega_k)  d\big(\frac{y_k}{\lambda}\big)\\
			&=& \frac{C_{-\Lambda}}{(\sqrt{2 \pi} \lambda)^N} \prod_{k=1}^{N} \int_{\mathbb{R}}  e^{-|y_k|} e^{i \big(c_k \big(\frac{y_k}{\lambda}\big)^2 + e_k (\frac{y_k}{\lambda})\big)} e^{-i(a \omega_k^2 + b \big(\frac{y_k}{\lambda}\big) \omega_k + c \big(\frac{y_k}{\lambda}\big)^2 + d \omega_k + e \big(\frac{y_k}{\lambda}\big))}  dy_k\\
			&=& \frac{C_{-\Lambda}}{(\sqrt{2 \pi} \lambda)^N} e_{-1}^{a, d}(\boldsymbol{\omega}) \prod_{k=1}^{N} \int_{\mathbb{R}}  e^{-|y_k|} e^{-i(b_k \big(\frac{y_k}{\lambda}\big) \omega_k)}  dy_k\\
			&=& \frac{1}{{\lambda}^N} e_{-1}^{a, d}(\boldsymbol{\omega}) e_{1}^{a, d}\big(\frac{\boldsymbol{\omega}}{\lambda}\big) \int_{\mathbb{R}^N} H_1(\boldsymbol{y}) e_{1}^{c, e}(\boldsymbol{y}) \mathcal{K}_{\boldsymbol{-\Lambda}}(\frac{\boldsymbol{\omega}}{\lambda}, \boldsymbol{y})d\boldsymbol{y}
			\\&=& \frac{e_{-1}^{a, d}(\boldsymbol{\omega}) e_{1}^{a, d}\big(\frac{\boldsymbol{\omega}}{\lambda}\big)}{{\lambda}^N}  h_{\boldsymbol{\Lambda}, 1} \big( \frac{\boldsymbol{\omega}}{\lambda}\big).
		\end{eqnarray*}
		This completes the proof.
	\end{proof}

	\begin{lemma}
		If $g \in L^\infty (\mathbb{R}^N)$ is a continuous function at $\boldsymbol{x}$, then $(g ~{\otimes}_{\boldsymbol{\Lambda}, 1}~ h_{\boldsymbol{\Lambda}, \lambda})(\boldsymbol{x}) \to g(\boldsymbol{x})$ as $\lambda$ tends to $0$. 
	\end{lemma}
	\begin{proof}
		We have, $h_{\boldsymbol{\Lambda}, \lambda} \in L^1(\mathbb{R}^N)$ and $g$ an element of $L^\infty (\mathbb{R}^N)$, by using the definition \ref{def3}, we get
		\begin{eqnarray*}
			(g ~{\otimes}_{\boldsymbol{\Lambda}, 1}~ h_{\boldsymbol{\Lambda}, \lambda})(\boldsymbol{x})
			&=& \frac{C_{\boldsymbol{\Lambda}}}{(\sqrt{2 \pi})^N} e^{a, d}_{-1}(\boldsymbol{x}) \int_{\mathbb{R}^N} ((e^{a, d}_{1}g)(\boldsymbol{x}-\boldsymbol{t}) (e^{a, d}_{1}h_{\boldsymbol{\Lambda}, \lambda})(\boldsymbol{t}) \big) d\boldsymbol{t}.
		\end{eqnarray*}		
		Applying Lemma \ref{lem1}-(4) gives 
		\begin{eqnarray*}
			(g ~{\otimes}_{\boldsymbol{\Lambda}, 1}~ h_{\boldsymbol{\Lambda}, \lambda})(\boldsymbol{x}) &=& \frac{C_{\boldsymbol{\Lambda}}}{(\sqrt{2 \pi})^N} e^{a, d}_{-1}(\boldsymbol{x}) \int_{\mathbb{R}^N} e^{a, d}_{1}(\boldsymbol{x}-\boldsymbol{t})g(\boldsymbol{x}-\boldsymbol{t}) e^{a, d}_{1}(\boldsymbol{t}) \frac{e_{-1}^{a, d}(\boldsymbol{t}) e_{1}^{a, d}\big(\frac{\boldsymbol{t}}{\lambda}\big)}{{\lambda}^N}  h_{\boldsymbol{\Lambda}, 1} \big( \frac{\boldsymbol{t}}{\lambda}\big) d\boldsymbol{t}
			\\ &=& \frac{C_{\boldsymbol{\Lambda}}}{(\sqrt{2 \pi})^N} e^{a,d}_{-1}(\boldsymbol{x}) \int_{\mathbb{R}^N} e^{a, d}_{1}(\boldsymbol{x}-\boldsymbol{t})g(\boldsymbol{x}-\boldsymbol{t})  \frac{e_{1}^{a, d}\big(\frac{\boldsymbol{t}}{\lambda}\big)}{{\lambda}^N}  h_{\boldsymbol{\Lambda}, 1} \big( \frac{\boldsymbol{t}}{\lambda}\big) d\boldsymbol{t}.
		\end{eqnarray*}	
		Take $\boldsymbol{u} = \frac{\boldsymbol{t}}{\lambda}$, we get
		\begin{eqnarray*}
			(g ~{\otimes}_{\boldsymbol{\Lambda}, 1}~ h_{\boldsymbol{\Lambda}, \lambda})(\boldsymbol{x}) &=& \frac{C_{\boldsymbol{\Lambda}}}{(\sqrt{2 \pi})^N} e^{a,d}_{-1}(\boldsymbol{x}) \int_{\mathbb{R}^N} e^{a, d}_{1}(\boldsymbol{x}-\lambda \boldsymbol{u})g(\boldsymbol{x}-\lambda \boldsymbol{u})  e_{1}^{a, d}\big(\boldsymbol{u} \big) h_{\boldsymbol{\Lambda}, 1} \big( \boldsymbol{u}\big) d\boldsymbol{u}.
		\end{eqnarray*} 
		The integrand converges to $\frac{C_{\boldsymbol{\Lambda}}}{(\sqrt{2 \pi})^N} g(\boldsymbol{x})  e_{1}^{a, d}\big(\boldsymbol{u} \big) h_{\boldsymbol{\Lambda}, 1} \big( \boldsymbol{u}\big)$ as $\lambda \to 0$, $\forall \boldsymbol{u} \in \mathbb{R}^N$ and is bounded by the integrable function $\big|\frac{C_{\boldsymbol{\Lambda}}}{(\sqrt{2 \pi})^N}\big| ||g||_\infty h_{\boldsymbol{\Lambda}, 1}$ for each $\lambda \neq 0$, we can now apply the dominated convergence theorem along with Lemma \ref{lem1}, leading to the integral tending to
		\begin{eqnarray*}
			\frac{C_{\boldsymbol{\Lambda}}}{(\sqrt{2 \pi})^N} \int_{\mathbb{R}^N} g(\boldsymbol{x})  e_{1}^{a, d}\big(\boldsymbol{u} \big) h_{\boldsymbol{\Lambda}, 1} \big( \boldsymbol{u}\big) d\boldsymbol{u} = g(\boldsymbol{x}) \text{ as } \lambda \to 0.
		\end{eqnarray*} 
	\end{proof}
	
	\begin{theorem}
		Let $1 \leq p < \infty$ and $f \in L^p(\mathbb{R}^N)$. Then $(f ~{\otimes}_{\boldsymbol{\Lambda}, 1}~ h_{\boldsymbol{\Lambda}, \lambda})(\boldsymbol{x}) \to f(\boldsymbol{x})$ in $L^p(\mathbb{R}^N)$ as $\lambda$ tends to $0$.
	\end{theorem}
	\begin{proof}
		For fixed $p \in [1, \infty)$, let $f \in L^p(\mathbb{R}^N)$. We define a measure on the Lebesgue $\sigma$-algebra $\mathcal{M}$ on $\mathbb{R}^N$ by
		\begin{eqnarray*}
			\mu(S) =  \frac{C_{\boldsymbol{\Lambda}}}{(\sqrt{2 \pi})^N} \int_{S} e_{1}^{a, d} \big(\boldsymbol{x} \big) h_{\boldsymbol{\Lambda}, \lambda} \big( \boldsymbol{x} \big) d\boldsymbol{x}, ~~~\forall S \in \mathcal{M}.
		\end{eqnarray*}
		We have, $\frac{C_{\boldsymbol{\Lambda}}}{(\sqrt{2 \pi})^N} e_{1}^{a, d} \big(\boldsymbol{x} \big) h_{\boldsymbol{\Lambda}, \lambda} \big( \boldsymbol{x}\big) \geq 0$, $\forall \boldsymbol{x} \in \mathbb{R}^N$. By the Lemma \ref{lem1}-(2), $\mu (\mathbb{R}^N) = 1$.
		Therefore, $\mu$ is a probability measure on $\mathcal{M}$. For a fixed $\boldsymbol{x} \in \mathbb{R}^N$, define
		\begin{eqnarray*}
			F_{\boldsymbol{x}}(\boldsymbol{t}) = \Big| (e^{a, d}_{1}f)(\boldsymbol{x}-\boldsymbol{t}) - (e^{a, d}_{1}f)(\boldsymbol{x}) \Big|, \forall \boldsymbol{t} \in \mathbb{R}^N.
		\end{eqnarray*} 
		Then by applying Holder's inequality and Lemma \ref{lem1}-(2), we get
		\begin{eqnarray*}
			&&\int_{\mathbb{R}^N} F_{\boldsymbol{x}}(\boldsymbol{t}) d\mu(\boldsymbol{t}) \\&=& \frac{C_{\boldsymbol{\Lambda}}}{(\sqrt{2 \pi})^N} \int_{\mathbb{R}^N} F_{\boldsymbol{x}}(\boldsymbol{t}) e_{1}^{a, d} \big(\boldsymbol{t} \big) h_{\boldsymbol{\Lambda}, \lambda} \big( \boldsymbol{t} \big) d\boldsymbol{t}
			\\&\leq& \frac{C_{\boldsymbol{\Lambda}}}{(\sqrt{2 \pi})^N} \Big( \int_{\mathbb{R}^N} \big| (e^{a, d}_{1}f)(\boldsymbol{x}-\boldsymbol{t}) \big|e_{1}^{a, d} \big(\boldsymbol{t} \big) h_{\boldsymbol{\Lambda}, \lambda} \big( \boldsymbol{t} \big) d\boldsymbol{t} + \int_{\mathbb{R}^N} \big| (e^{a, d}_{1}f)(\boldsymbol{x}) \big| e_{1}^{a, d} \big(\boldsymbol{t} \big) h_{\boldsymbol{\Lambda}, \lambda} \big( \boldsymbol{t} \big) d\boldsymbol{t} \Big)
			\\&=& \frac{C_{\boldsymbol{\Lambda}}}{(\sqrt{2 \pi})^N}  \int_{\mathbb{R}^N} \big|f(\boldsymbol{x}-\boldsymbol{t}) \big| \big(e_{1}^{a, d} (\boldsymbol{t}) h_{\boldsymbol{\Lambda}, \lambda} (\boldsymbol{t}) \big) d\boldsymbol{t} + |f(\boldsymbol{x})|
			\\&\leq& \frac{C_{\boldsymbol{\Lambda}}}{(\sqrt{2 \pi})^N} \Big(\int_{\mathbb{R}^N} \big|f(\boldsymbol{x}-\boldsymbol{t}) \big|^p d\boldsymbol{t} \Big)^\frac{1}{p} \Big(\int_{\mathbb{R}^N} \big|e_{1}^{a, d} (\boldsymbol{t}) h_{\boldsymbol{\Lambda}, \lambda} (\boldsymbol{t}) \big|^q d\boldsymbol{t} \Big)^\frac{1}{q}  + |f(\boldsymbol{x})|
			\\&=& \frac{C_{\boldsymbol{\Lambda}}}{(\sqrt{2 \pi})^N} ||f||_p ||h_{\boldsymbol{\Lambda}, \lambda}||_q + |f(\boldsymbol{x})|
			~~<~ \infty.
		\end{eqnarray*}
		i.e.,
		\begin{eqnarray*}
			\int_{\mathbb{R}^N} F_{\boldsymbol{x}}(\boldsymbol{t}) d\mu(\boldsymbol{t}) < \infty.
		\end{eqnarray*}
		Thus $F_{\boldsymbol{x}}$ is an integrable function on $\mathbb{R}^2$ with respect to measure $\mu$.
		\\Let $\phi(r) = r^p$. Then 
		\begin{eqnarray*}
			\phi'' (r) = p(p-1)r^{p-2} \geq 0 \text{ on } [0, \infty).
		\end{eqnarray*}
		Thus $\phi$ is a convex function on $[0, \infty)$. Hence, by applying Jensen's inequality \cite{Rudin},
		\begin{eqnarray*}
			\phi\Big( \int_{\mathbb{R}^N} F_{\boldsymbol{x}}(\boldsymbol{t}) d\mu(\boldsymbol{t}) \Big) \leq  \int_{\mathbb{R}^N} (\phi \circ F_{\boldsymbol{x}})(\boldsymbol{t}) d\mu(\boldsymbol{t}). 
		\end{eqnarray*}
		i.e.,
		\begin{eqnarray}\label{f}
			\Big( \int_{\mathbb{R}^N} F_{\boldsymbol{x}}(\boldsymbol{t}) d\mu(\boldsymbol{t}) \Big)^p \leq  \int_{\mathbb{R}^N} (F_{\boldsymbol{x}} (\boldsymbol{t}))^p d\mu(\boldsymbol{t}).
		\end{eqnarray}
		By using Lemmas \ref{lem1} and \eqref{f}, we get
		\begin{eqnarray*}
			&&|(f ~{\otimes}_{\boldsymbol{\Lambda}, 1}~ h_{\boldsymbol{\Lambda}, \lambda})(\boldsymbol{x}) - f(\boldsymbol{x})|^p \\&=& \bigg|\bigg(\frac{C_{\boldsymbol{\Lambda}}}{(\sqrt{2 \pi})^N} e^{a, d}_{-1}(\boldsymbol{x}) \int_{\mathbb{R}^N} \big((e^{a, d}_{1}f)(\boldsymbol{x}-\boldsymbol{t}) (e^{a, d}_{1}h_{\boldsymbol{\Lambda}, \lambda})(\boldsymbol{t}) \big) d\boldsymbol{t}\bigg) - f(\boldsymbol{x})	\frac{C_{\boldsymbol{\Lambda}}}{(\sqrt{2 \pi})^N} \int_{\mathbb{R}^N} (e_{1}^{a, d} h_{\boldsymbol{\Lambda}, \lambda}) ( \boldsymbol{t} ) d\boldsymbol{t} \bigg|^p
			\\&\leq& \bigg( \frac{|C_{\boldsymbol{\Lambda}|}}{(\sqrt{2 \pi})^N} \int_{\mathbb{R}^N} \Big| (e^{a, d}_{1}f)(\boldsymbol{x}-\boldsymbol{t}) - (e^{a, d}_{1}f)(\boldsymbol{x}) \Big|  (e_{1}^{a, d} h_{\boldsymbol{\Lambda}, \lambda} )( \boldsymbol{t}) d\boldsymbol{t} \bigg)^p
			\\&=& \bigg( \frac{|C_{\boldsymbol{\Lambda}|}}{(\sqrt{2 \pi})^N} \int_{\mathbb{R}^N} F_{\boldsymbol{x}}(\boldsymbol{t})  (e_{1}^{a, d} h_{\boldsymbol{\Lambda}, \lambda} )(\boldsymbol{t}) d\boldsymbol{t} \bigg)^p
			\\&=& \Big( \int_{\mathbb{R}^N} F_{\boldsymbol{x}}(\boldsymbol{t})d\mu(\boldsymbol{t}) \Big)^p
			\\&\leq& \int_{\mathbb{R}^N} \Big(  F_{\boldsymbol{x}}(\boldsymbol{t}) \Big)^p d\mu(\boldsymbol{t})
			\\&=& \frac{|C_{\boldsymbol{\Lambda}}|}{(\sqrt{2 \pi})^N} \int_{\mathbb{R}^N} \Big| (e^{a, d}_{1}f)(\boldsymbol{x}-\boldsymbol{t}) - (e^{a, d}_{1}f)(\boldsymbol{x}) \Big|^p  (e_{1}^{a, d} h_{\boldsymbol{\Lambda}, \lambda} )(\boldsymbol{t})  d\boldsymbol{t}. 		
		\end{eqnarray*} 
		Rewriting the above expression using \ref{def4},  $(\tau_{\boldsymbol{t}} f) (\boldsymbol{x}) = f(\boldsymbol{x}-\boldsymbol{t})$, $\forall \boldsymbol{x} \in \mathbb{R}^N$, we get
		\begin{eqnarray*}
			|(f ~{\otimes}_{\boldsymbol{\Lambda}, 1}~ h_{\boldsymbol{\Lambda}, \lambda})(\boldsymbol{x}) - f(x)|^p &\leq& \frac{|C_{\boldsymbol{\Lambda}|}}{(\sqrt{2 \pi})^N} \int_{\mathbb{R}^N} \Big| \big(\tau_{\boldsymbol{t}}(e^{a, d}_{1}f)\big)(\boldsymbol{x}) - (e^{a, d}_{1}f)(\boldsymbol{x}) \Big|^p  (e_{1}^{a, d} h_{\boldsymbol{\Lambda}, \lambda} )(\boldsymbol{t})  d\boldsymbol{t}. 
		\end{eqnarray*}
		Multiplying by $\frac{|C_{\boldsymbol{\Lambda}}|}{(\sqrt{2 \pi})^N}$ and integrating with respect to $\boldsymbol{x}$ gives
		\begin{eqnarray*}
			&& \frac{|C_{\boldsymbol{\Lambda}}|}{(\sqrt{2 \pi})^N}\int_{\mathbb{R}^N}|(f ~{\otimes}_{\boldsymbol{\Lambda}, 1}~ h_{\boldsymbol{\Lambda}, \lambda})(\boldsymbol{x}) - f(x)|^p d\boldsymbol{x} 
			\\&\leq& \bigg( \frac{|C_{\boldsymbol{\Lambda}}|}{(\sqrt{2 \pi})^N}\bigg)^2 \int_{\mathbb{R}^N} \int_{\mathbb{R}^N} \Big| \big(\tau_{\boldsymbol{t}}(e^{a, d}_{1}f)\big)(\boldsymbol{x}) - (e^{a, d}_{1}f)(\boldsymbol{x}) \Big|^p  (e_{1}^{a, d} h_{\boldsymbol{\Lambda}, \lambda})(\boldsymbol{t})  d\boldsymbol{t} d\boldsymbol{x}
			\\&=& \frac{|C_{\boldsymbol{\Lambda}}|}{(\sqrt{2 \pi})^N} \int_{\mathbb{R}^N} \big| \big| \big(\tau_{\boldsymbol{t}}(e^{a, d}_{1}f)\big)(\boldsymbol{x}) - (e^{a, d}_{1}f)(\boldsymbol{x})\big| \big|^p_p (e_{1}^{a, d} h_{\boldsymbol{\Lambda}, \lambda})(\boldsymbol{t}) d\boldsymbol{t}. 
		\end{eqnarray*}
		Let
		\begin{eqnarray*}
			g(\boldsymbol{t}) = \big| \big| \big(\tau_{-\boldsymbol{t}}(e^{a, d}_{1}f)\big)(\boldsymbol{x}) - (e^{a, d}_{1}f)(\boldsymbol{x})\big| \big|^p_p e^{a, d}_{1} (\boldsymbol{t}).
		\end{eqnarray*}
		Then
		\begin{eqnarray*}
			(ge^{a, d}_{-1})(\boldsymbol{-t}) &=& \big| \big| \big(\tau_{\boldsymbol{t}}(e^{a, d}_{1}f)\big)(\boldsymbol{x}) - (e^{a, d}_{1}f)(\boldsymbol{x})\big| \big|^p_p.
		\end{eqnarray*}
		Therefore
		\begin{eqnarray*}
			\int_{\mathbb{R}^N}|(f ~{\otimes}_{\boldsymbol{\Lambda}, 1}~ h_{\boldsymbol{\Lambda}, \lambda})(\boldsymbol{x}) - f(x)|^p d\boldsymbol{x} 
			&\leq& \frac{|C_{\boldsymbol{\Lambda}}|}{(\sqrt{2 \pi})^N} \int_{\mathbb{R}^N} (ge^{a, d}_{-1})(\boldsymbol{-t}) (e_{1}^{a, d} h_{\boldsymbol{\Lambda}, \lambda})(\boldsymbol{t}) d\boldsymbol{t}
			\\&=& (g \otimes_{\tiny \boldsymbol{\Lambda}, \lambda} h_{\tiny \boldsymbol{\Lambda}, \lambda}) (\boldsymbol{0}). 
		\end{eqnarray*}
		Since $g(\boldsymbol{t})$ is continuous and satisfies the bound $(2||f||_p)^p$. By Lemma \ref{lem1}, $h_{\tiny \boldsymbol{\Lambda}, \lambda}(\boldsymbol{t})$ is tending to $0$, then $g(\boldsymbol{0})$ is zero and $f \otimes_{\boldsymbol{\Lambda}, \lambda} h \to f$. Thus, the proof of the theorem is complete.
	\end{proof}
	\begin{theorem}($L^1$- inversion theorem) \label{inv1}
		Let $\mathcal{Q}_{\Lambda}f \in L^1(\mathbb{R}^N)$. Then 
		\begin{eqnarray*}
			f(\boldsymbol{x}) = \int_{\mathbb{R}^N} (\mathcal{Q}_{\boldsymbol{\Lambda}} f) (\boldsymbol{t}) \mathcal{K}_{\boldsymbol{-\Lambda}} (\boldsymbol{t}, \boldsymbol{\omega}) d\boldsymbol{t}.
		\end{eqnarray*}
		The above analogy can be written as
		\begin{eqnarray*}
			(\mathcal{Q}_{-\boldsymbol{\Lambda}} \circ \mathcal{Q}_{\boldsymbol{\Lambda}})(f) = f.
		\end{eqnarray*}
		\begin{proof}
			Let $f$ be in $L^1(\mathbb{R}^N)$ and $h_{\boldsymbol{\Lambda}, \lambda}$ be defined as in Lemma \ref{lem1} for $\lambda > 0$. Then, we obtain
			\begin{eqnarray*}
				&&(f ~{\otimes}_{\boldsymbol{\Lambda}, 1}~ h_{\boldsymbol{\Lambda}, \lambda})(\boldsymbol{x}) 
				\\&=& \frac{C_{\boldsymbol{\Lambda}}}{(\sqrt{2 \pi})^N} e^{a, d}_{-1}(\boldsymbol{x}) \int_{\mathbb{R}^N} ((e^{a, d}_{1}f)(\boldsymbol{x}-\boldsymbol{y}) (e^{a, d}_{1})(\boldsymbol{y})\Big( \int_{\mathbb{R}^N} H_\lambda(\boldsymbol{t}) e_1^{\boldsymbol{c, e}} (\boldsymbol{t}) \mathcal{K}_{-\boldsymbol{\Lambda}} (\boldsymbol{t}, \boldsymbol{y}) d\boldsymbol{t}\Big)  d\boldsymbol{y}
				\\&=& \frac{C_{\boldsymbol{\Lambda}}C_{\boldsymbol{-\Lambda}}}{(2 \pi)^N} e^{a, d}_{-1}(\boldsymbol{x}) \int_{\mathbb{R}^N} (e^{a, d}_{1}f)(\boldsymbol{x}-\boldsymbol{y}) \Big(\int_{\mathbb{R}^N} H_\lambda(\boldsymbol{t}) e^{-i\sum_{k=1}^{N}b_k y_k t_k} d\boldsymbol{t}\Big) d\boldsymbol{y}.
			\end{eqnarray*}
			Since $f$ and $H_{\lambda}(\boldsymbol{x})$ belong to $L^1(\mathbb{R}^N)$, we can apply the Fubini's theorem, which gives us
			\begin{eqnarray*}
				(f ~{\otimes}_{\boldsymbol{\Lambda}, 1}~ h_{\boldsymbol{\Lambda}, \lambda})(\boldsymbol{x}) 
				&=& \frac{C_{\boldsymbol{\Lambda}}C_{\boldsymbol{-\Lambda}}}{(2 \pi)^N} e^{a, d}_{-1}(\boldsymbol{x}) \int_{\mathbb{R}^N} H_\lambda(\boldsymbol{t}) \Big( \int_{\mathbb{R}^N} (e^{a,d}_{1}f)(\boldsymbol{x}-\boldsymbol{y})  e^{-i\sum_{k=1}^{N}b_k y_k t_k} d\boldsymbol{y}\Big)  d\boldsymbol{t}.
			\end{eqnarray*}	
			Substituting $\boldsymbol{z}= \boldsymbol{x}- \boldsymbol{y}$, we arrive at
			\begin{eqnarray*}
				(f ~{\otimes}_{\boldsymbol{\Lambda}, 1}~ h_{\boldsymbol{\Lambda}, \lambda})(\boldsymbol{x}) 
				&=& \frac{C_{\boldsymbol{\Lambda}}C_{\boldsymbol{-\Lambda}}}{(2 \pi)^N} e^{a, d}_{-1}(\boldsymbol{x})  \int_{\mathbb{R}^N} H_\lambda(\boldsymbol{t}) e^{-i\sum_{k=1}^{N}b_k x_k t_k} \Big(\int_{\mathbb{R}^N} (e^{a, d}_{1}f)(\boldsymbol{z})  e^{i\sum_{k=1}^{N}b_k z_k t_k} d\boldsymbol{z}\Big)  d\boldsymbol{t}
				\\&=& \frac{C_{\boldsymbol{-\Lambda}}}{(2 \pi)^N} e^{a, d}_{-1}(\boldsymbol{x})  \int_{\mathbb{R}^N} H_\lambda(\boldsymbol{t}) e^{-i\sum_{k=1}^{N}b_k x_k t_k} \Big(\int_{\mathbb{R}^N} (e^{c, e}_{-1})(\boldsymbol{t}) f(\boldsymbol{z}) \mathcal{K}_{\boldsymbol{\Lambda}}(\boldsymbol{t}, \boldsymbol{z}) d\boldsymbol{z}\Big)  d\boldsymbol{t}
				\\&=& \frac{C_{\boldsymbol{-\Lambda}}}{(2 \pi)^N} e^{a, d}_{-1}(\boldsymbol{x})  \int_{\mathbb{R}^N} H_\lambda(\boldsymbol{t}) e^{-i\sum_{k=1}^{N}b_k x_k t_k} (e^{c, e}_{-1})(\boldsymbol{t}) \mathcal{Q}_{\boldsymbol{\Lambda}}(f)(\boldsymbol{t})d\boldsymbol{t}
				\\&=& \int_{\mathbb{R}^N} H_\lambda(\boldsymbol{t}) \mathcal{Q}_{\boldsymbol{\Lambda}}(f)(\boldsymbol{t}) \mathcal{K}_{\boldsymbol{-\Lambda}}(\boldsymbol{t}, \boldsymbol{x}) d\boldsymbol{t}
				\\&\to& \int_{\mathbb{R}^N} \mathcal{Q}_{\boldsymbol{\Lambda}}(f)(\boldsymbol{t}) \mathcal{K}_{\boldsymbol{-\Lambda}}(\boldsymbol{t}, \boldsymbol{x})d\boldsymbol{t}, \text{ as } \lambda \to 0,
			\end{eqnarray*}
			since by Lemma \ref{lem1}-(1), we have $H_\lambda(\boldsymbol{t})$ tends to $1$ as $\lambda$ tends to $0$ for each $\boldsymbol{t} \in \mathbb{R}^N$, applying the dominated convergence theorem gives
			\begin{eqnarray*}
				(f ~{\otimes}_{\boldsymbol{\Lambda}, 1}~ h_{\boldsymbol{\Lambda}, \lambda})(\boldsymbol{x}) \to f(\boldsymbol{x}).
			\end{eqnarray*}
			Hence we obtain the required result.
		\end{proof}
	\end{theorem}
	\begin{theorem}($L^2$- inversion theorem)
		Let $f$ be a function in $L^2 (\mathbb{R}^N)$. Then 
		\begin{eqnarray*}
			(\mathcal{Q}_{\boldsymbol{-\Lambda}} \circ \mathcal{Q}_{\boldsymbol{\Lambda}})(f) = f.	
		\end{eqnarray*}
	\end{theorem}
	\begin{proof}
		For $f \in L^1(\mathbb{R}^N) \cap L^2(\mathbb{R}^N)$, since $h_{\boldsymbol{\Lambda}, \lambda} \in L^p(\mathbb{R}^N)$ for $p \in [1, \infty)$, it follows that $f \otimes_{\boldsymbol{\Lambda}, 1} h_{\boldsymbol{\Lambda}, \lambda} \in L^1(\mathbb{R}^N)$. Moreover, since $\mathcal{Q}_{\boldsymbol{\Lambda}}f$ and $\mathcal{Q}_{\boldsymbol{\Lambda}} h_{\boldsymbol{\Lambda}, \lambda}$ belong to $L^2(\mathbb{R}^N)$, their product is in $L^1(\mathbb{R}^N)$. Applying the Theorem \ref{conv2}, we obtain
		\begin{eqnarray*}
			\mathcal{Q}_{\Lambda} (f \otimes_{\boldsymbol{\Lambda}, 1} h_{\boldsymbol{\Lambda}, \lambda})(\boldsymbol{\omega}) = e_{-1}^{c,e}(\boldsymbol{\omega}) \mathcal{Q}_{\boldsymbol{\Lambda}}(f)(\boldsymbol{\omega}) \mathcal{Q}_{\boldsymbol{\Lambda}}(h_{\boldsymbol{\Lambda}, \lambda})(\boldsymbol{\omega}).
		\end{eqnarray*}
		Thus $\mathcal{Q}_{\Lambda} (f \otimes_{\boldsymbol{\Lambda}, 1} h_{\boldsymbol{\Lambda}, \lambda})$ belongs to $L^1(\mathbb{R}^N) \cap L^2(\mathbb{R}^N)$. Furthermore, by applying Theorem \ref{inv1}, we get
		\begin{eqnarray*}
			(\mathcal{Q}_{-\Lambda} \circ \mathcal{Q}_{\Lambda}) (f \otimes_{\boldsymbol{\Lambda}, 1} h_{\boldsymbol{\Lambda}, \lambda}) = (f \otimes_{\boldsymbol{\Lambda}, 1} h_{\boldsymbol{\Lambda}, \lambda}).
		\end{eqnarray*}
		We have, $f \otimes_{\boldsymbol{\Lambda}, 1} h_{\boldsymbol{\Lambda}, \lambda} \to f \in L^2(\mathbb{R}^N)$ as $\lambda \to 0$ and $\mathcal{Q}_{\boldsymbol{-\Lambda}} \circ \mathcal{Q}_{\boldsymbol{\Lambda}}$ is continuous on $L^2(\mathbb{R}^N)$. Hence,
		\begin{eqnarray*}
			(\mathcal{Q}_{-\Lambda} \circ \mathcal{Q}_{\Lambda}) (f \otimes_{\boldsymbol{\Lambda}, 1} h_{\boldsymbol{\Lambda}, \lambda}) \to (\mathcal{Q}_{-\Lambda} \circ \mathcal{Q}_{\Lambda}) (f) \in L^(\mathbb{R}^N) \text{ as } \lambda \to 0. 
		\end{eqnarray*}
		Consequently
		\begin{eqnarray*}
			(\mathcal{Q}_{-\Lambda} \circ \mathcal{Q}_{\Lambda}) (f) = f,~~ \forall f \in L^1(\mathbb{R}^N) \cap L^2(\mathbb{R}^N). 
		\end{eqnarray*}
		By again employing the continuity of $	(\mathcal{Q}_{-\Lambda} \circ \mathcal{Q}_{\Lambda})$ on $L^(\mathbb{R}^N)$, and the density of $L^1(\mathbb{R}^N) \cap L^2(\mathbb{R}^N)$, we deduce
		\begin{eqnarray*}
			(\mathcal{Q}_{-\Lambda} \circ \mathcal{Q}_{\Lambda}) (f) = f,~~ \forall f \in L^2(\mathbb{R}^N).
		\end{eqnarray*}
		Hence, we get the desired result.
	\end{proof}
	\begin{theorem}
		Let $f$ be in $L^2(\mathbb{R}^N)$ and $\mathcal{Q}_{\boldsymbol{\Lambda}}f$ be in  $L^1(\mathbb{R}^N)$. Then 
		\begin{eqnarray*}
			f(\boldsymbol{x}) = \int_{\mathbb{R}^N} (\mathcal{Q}_{\boldsymbol{\Lambda}} f) (\boldsymbol{t}) \mathcal{K}_{-\Lambda}(\boldsymbol{t}, \boldsymbol{x}) d\boldsymbol{t}.
		\end{eqnarray*}
	\end{theorem}
	\begin{proof}
		The proof proceeds similarly to \ref{inv1}.
	\end{proof}
	\begin{theorem}(Product theorem)
		Let $f , g \in L^2(\mathbb{R}^N)$ and $\mathcal{Q}_{\boldsymbol{\Lambda'}}f, \mathcal{Q}_{\boldsymbol{\Lambda'}}g \in L^1(\mathbb{R}^N)$, then
		\begin{eqnarray*}
			(\mathcal{Q}_{\boldsymbol{\Lambda'}}f) \otimes_{\boldsymbol{-\Lambda}, \lambda} (\mathcal{Q}_{\boldsymbol{\Lambda'}}g) = \mathcal{Q}_{\boldsymbol{\Lambda'}}(e^{c, e}_{\lambda^2} fg).
		\end{eqnarray*}
	\end{theorem}	
	\begin{proof}
		Invoking Theorem \ref{conv2}, we get
		\begin{eqnarray*}
			\mathcal{Q}_{-\boldsymbol{\Lambda'}}\big[(\mathcal{Q}_{\boldsymbol{\Lambda'}} f) \otimes_{-\boldsymbol{\Lambda'},\lambda} (\mathcal{Q}_{\boldsymbol{\Lambda'}} g)\big] &=& e^{c, e}_{\lambda^2} \mathcal{Q}_{-\boldsymbol{\Lambda'}}\big[(\mathcal{Q}_{\boldsymbol{\Lambda'}} f)\big] \mathcal{Q}_{-\boldsymbol{\Lambda'}}\big[(\mathcal{Q}_{\boldsymbol{\Lambda'}} g)\big]
			\\&=& e^{c, e}_{\lambda^2} fg.
		\end{eqnarray*}
		Applying multidimensional QPFT $\mathcal{Q}_{\boldsymbol{\Lambda'}}$ on both sides of the last equation, we get
		\begin{eqnarray*}
			(\mathcal{Q}_{\boldsymbol{\Lambda'}} f) \otimes_{-\boldsymbol{\Lambda'},\lambda} (\mathcal{Q}_{\boldsymbol{\Lambda'}} g) &=& \mathcal{Q}_{\boldsymbol{\Lambda'}} \big(e^{c, e}_{\lambda^2} fg\big).
		\end{eqnarray*}
		Therefore, the statement holds.
	\end{proof}

	\section{Boas theorem for multidimensional QPFT}
	In this Section, we discuss analogue of Boas' theorem for multidimensional QPFT, for which we first prove the following lemma.
	\begin{lemma}\label{lem2}
		Let $f \in L^1 (\mathbb{R}^N)$ be an infinitely differentiable function and $\mathcal{Q}_{\boldsymbol{\Lambda}} f$ be its N-dimensional quadratic phase Fourier transform. Then 
		\begin{eqnarray}\label{4.1}
			(\mathcal{Q}_{\boldsymbol{\Lambda}} \Delta_{\boldsymbol{x}} ^n f)(\boldsymbol{\omega}) &=& (ib_1 \omega_1)^n (ib_2 \omega_2)^n . . . (ib_N \omega_N)^n (\mathcal{Q}_{\boldsymbol{\Lambda}} f)(\boldsymbol{\omega})
			\\&=& \prod_{k=1}^{N} (ib_k \omega_k)^n (\mathcal{Q}_{\boldsymbol{\Lambda}} f)(\boldsymbol{\omega}),
		\end{eqnarray}
		where
		\begin{eqnarray}
			\Delta_{\boldsymbol{x}} &=& (-1)^N\prod_{k=1}^{N} \Big( \frac{\partial{}}{\partial{x_k}} + i(2a_k x_k +d_k) \Big) = (-1)^N \prod_{k=1}^{N} \Delta_{x_k}
		\end{eqnarray}
		and
		\begin{eqnarray}
			\Delta_{x_k}^n=(-1)^n \sum_{m=0}^{n} P_m (x_k) D^{n-m}_{x_k}, 
		\end{eqnarray}
		where
		\begin{eqnarray*}
			P_m(x_k)=\binom n{x_k} (i)^m \sum_{r=0}^{m} \binom mr (2ax_k)^r d^{m-r}. 
		\end{eqnarray*}
		Here $P_m$ denotes the $m^{th}$ order polynomial with $P_0 (x_k)=1$.
	\end{lemma}	
	
	\begin{proof}
		Let $\Delta_{x_k}^* = \frac{\partial{}}{\partial{x_k}} - i(2a_kx_k +d_k)$.
		Then the kernel of the N-dimensional QPFT satisfies the following
		\begin{eqnarray}
			(\Delta_{x_k}^*)\mathcal{K}_{\Lambda_k} f(\omega_k,x_k)= (ib_k\omega_k) \mathcal{K}_{\Lambda_k} (\omega_k,x_k). \label{4.4}
		\end{eqnarray}
		For $n \in \mathbb{N}$,
		\begin{eqnarray*}
			(\Delta_{x_k}^*)^n\mathcal{K}_{\Lambda_k}(\omega_k,x_k)&=& (ib_k\omega_k)^n \mathcal{K}_{\Lambda_k}(\omega_k,x_k). \\
			\int_{\mathbb{R}}(\Delta_{x_k}^*)^n\mathcal{K}_{\Lambda_k}(\omega_k,x_k) f(x_k)dx_k &=& \int_{\mathbb{R}}\mathcal{K}_{\Lambda_k}(\omega_k,x_k)(\Delta_{x_k})^n f(x_k)dx_k.
		\end{eqnarray*}
		\begin{eqnarray}
			(\Delta_{x_k}^*)^n\mathcal{K}_{\boldsymbol{\Lambda}}(\boldsymbol{\omega},\boldsymbol{x})&=& (ib_k\omega_k)^n \mathcal{K}_{\boldsymbol{\Lambda}}(\boldsymbol{\omega}, \boldsymbol{x}).\label{4.5} \\
			\int_{\mathbb{R}^N}(\Delta_{x_k}^*)^n\mathcal{K}_{\boldsymbol{\Lambda}}(\boldsymbol{\omega}, \boldsymbol{x}) f(\boldsymbol{x})d\boldsymbol{x} &=& \int_{\mathbb{R}^N}\mathcal{K}_{\boldsymbol{\Lambda}}(\boldsymbol{\omega}, \boldsymbol{x})(\Delta_{x_k})^n f(\boldsymbol{x})d\boldsymbol{x}.\label{4.6}
		\end{eqnarray}
		Using \eqref{4.4},\eqref{4.5} and \eqref{4.6}, we get 
		\begin{eqnarray*}
			(\mathcal{Q}_{\boldsymbol{\Lambda}} \Delta_{x_k} ^n f)(\boldsymbol{\omega}) &=&
			\int_{\mathbb{R}^N} \mathcal{K}_{\boldsymbol{\Lambda}}(\boldsymbol{\omega}, \boldsymbol{x}) \Delta_{x_k}^n f(\boldsymbol{x}) d\boldsymbol{x} \\
			&=& \int_{\mathbb{R}^N}[(\Delta_{x_k}^*)^n \mathcal{K}_{\boldsymbol{\Lambda}}(\boldsymbol{\omega}, \boldsymbol{x})] f(\boldsymbol{x}) d\boldsymbol{x} \\
			&=& \int_{\mathbb{R}^N}[(ib_k\omega_k)^n \mathcal{K}_{\boldsymbol{\Lambda}} (\boldsymbol{\omega}, \boldsymbol{x})] f(\boldsymbol{x}) d\boldsymbol{x} \\
			&=& (ib_k\omega_k)^n(\mathcal{Q}_{\boldsymbol{\Lambda}} f)(\boldsymbol{\omega}).
		\end{eqnarray*}
		\begin{eqnarray*}
			(\mathcal{Q}_{\boldsymbol{\Lambda}} \Delta_{\boldsymbol{x}} ^n f)(\boldsymbol{\omega}) &=& (\mathcal{Q}_{\boldsymbol{\Lambda}} (\Delta_{x_1} ^n f)(\Delta_{x_2} ^n f) . . . (\Delta_{x_k} ^n) f)(\boldsymbol{\omega})\\
			&=& (ib_1\omega_1)^n(\mathcal{Q}_{\boldsymbol{\Lambda}} (\Delta_{x_2} ^n f) . . . (\Delta_{x_N} ^n) f)(\boldsymbol{\omega})\\
			&=& (ib_1\omega_1)^n (ib_2\omega_2)^n(\mathcal{Q}_{\boldsymbol{\Lambda}} (\Delta_{x_3} ^n f) . . . (\Delta_{x_N} ^n) f)(\boldsymbol{\omega})\\
			&=& \prod_{k=1}^{N} (ib_k \omega_k)^n (\mathcal{Q}_{\boldsymbol{\Lambda}} f)(\boldsymbol{\omega}).
		\end{eqnarray*}
		Hence we obtain the required result.
	\end{proof}	
	\begin{theorem} (Boas theorem for multidimensional QPFT)
		Let the function $f \in L^1(\mathbb{R}^N)$. Then the QPFT of $f$ vanishes in the neighborhood of the origin if and only if $B^n f$ is well-defined and belongs to $L^1(\mathbb{R}^N)$, $\forall n \in \mathbb{Z}_+$ and 
		\begin{eqnarray}
			\lim_{n \to \infty} ||B^nf||^\frac{1}{n}= R < \infty,
		\end{eqnarray}
		where
		\begin{eqnarray*}
			(Bf)(\boldsymbol{x})&=&e^{-\sum_{k=1}^{N}i(ax_k^2+dx_k)}\int_{x_1}^{\infty} \int_{x_2}^{\infty}. . . \int_{x_N}^{\infty} e^{\sum_{k=1}^{N}i(at_k^2+dt_k)}f(\boldsymbol{t})d\boldsymbol{t},
		\end{eqnarray*}
		$R=\gamma^{-1} \hspace{2mm} and \hspace{2mm} \gamma_k = \inf \{{|b_k \omega_k|: \boldsymbol{\omega} = (\omega_1, \omega_2, . . ., \omega_N)\in\sup\mathcal{Q}^{a, b, c}_{d, e} f}\}$, $\gamma = \prod_{k=1}^{N} \gamma_k$. 
	\end{theorem}
	
	\begin{proof}(Necessary condition)
		Suppose	$\mathcal{Q}_{\boldsymbol{\Lambda}}(\boldsymbol{\omega})$ vanishes on $\big(-\frac{\gamma_1}{|b_1|},\frac{\gamma_1}{|b_1|}\big)  \times\big(-\frac{\gamma_2}{|b_2|},\frac{\gamma_2}{|b_2|}\big) \times . . . \times \big(-\frac{\gamma_N}{|b_N|},\frac{\gamma_N}{|b_N|}\big)$ such that
		\begin{eqnarray}
			(B_{\boldsymbol{m}}f)(\boldsymbol{x}) \hspace{-0.2cm} &=& \hspace{-0.2cm} e^{\sum_{k=1}^{N}-i(ax_k^2+dx_k)}\int_{x_1}^{m_1}\int_{x_2}^{m_2}. . . \int_{x_N}^{m_N} e^{\sum_{k=1}^{N}i(at_k^2+dt_k)}f(\boldsymbol{t})d\boldsymbol{t},
		\end{eqnarray}
		where $\boldsymbol{m}= (m_1, m_2, . . ., m_N)$.
		Obviously
		\begin{eqnarray}
			\lim_{\boldsymbol{m} \to \infty} (B_{\boldsymbol{m}}f)(\boldsymbol{x})&=& (Bf)(\boldsymbol{x}). 
		\end{eqnarray}
		Define $\chi_{\boldsymbol{x}}^{\boldsymbol{m}}(\boldsymbol{t}) = \prod_{k=1}^{N} \chi_{x_k}^{m_k}(t)$, where the characteristic function $\chi_{x_k}^{m_k}(t)$ is given by
		\begin{eqnarray*}
			\chi_{x_k}^{m_k}(t)=
			\begin{cases}
				1 & , \hspace{3mm} \text{if }  t \in [x_k, m_k]\\
				0 & ,\hspace{3mm} \text{elsewhere }. 
			\end{cases}
		\end{eqnarray*}
		Invoking \ref{def1} and \ref{nm}, we obtain
		\begin{eqnarray*}
			&&	(B_{\boldsymbol{m}}f)(\boldsymbol{x})
			\\&=&e^{\sum_{k=1}^{N}-i(a_kx_k^2+d_kx_k)}\int_{x_1}^{m_1}\int_{x_2}^{m_2}. . . \int_{x_N}^{m_N} e^{\sum_{k=1}^{N}i(a_kt_k^2+d_kt_k)}f(\boldsymbol{t})d\boldsymbol{t}
			\\&=& e^{\sum_{k=1}^{N}-i(a_kx_k^2+d_kx_k)}\int_{\mathbb{R}^N}  e^{\sum_{k=1}^{N}i(a_kt_k^2+d_kt_k)}f(\boldsymbol{t})\chi_{\boldsymbol{x}}^{\boldsymbol{m}}(\boldsymbol{t})d\boldsymbol{t}
			\\&=& e^{\sum_{k=1}^{N}-i(a_kx_k^2+d_kx_k)}\big< e^{\sum_{k=1}^{N}i(a_kt_k^2+d_kt_k)}f(\boldsymbol{t}), \chi_{\boldsymbol{x}}^{\boldsymbol{m}}(\boldsymbol{t}) \big>
			\\&=& e^{\sum_{k=1}^{N}-i(a_kx_k^2+d_kx_k)} \big< (\mathcal{Q}_{\boldsymbol{\Lambda}} f)(\boldsymbol{\omega}), (\mathcal{Q}_{\boldsymbol{\Lambda}} e^{\sum_{k=1}^{N}-i(a_kt_k^2+d_kt_k)} \chi_{\boldsymbol{x}}^{\boldsymbol{m}})(\boldsymbol{\omega}) \big>
			\\&=& e^{\sum_{k=1}^{N}-i(a_kx_k^2+d_kx_k)} \int_{\mathbb{R}^N} (\mathcal{Q}_{\boldsymbol{\Lambda}} f)(\boldsymbol{\omega}) \overline{\int_{x_1}^{m_1} \int_{x_2}^{m_2} . . . \int_{x_N}^{m_N} \mathcal{K}_{\boldsymbol{\Lambda}} e^{\sum_{k=1}^{N}-i(a_kt_k^2+d_kt_k)} \chi_{\boldsymbol{x}}^{\boldsymbol{m}} (\boldsymbol{t})d\boldsymbol{t}} d\boldsymbol{\omega}
			\\&=&\overline{\sqrt{\frac{\prod_{k=1}^{N}b_k}{(2\pi i)^N}}} e^{\sum_{k=1}^{N}-i(a_kx_k^2+d_kx_k)} \int_{\mathbb{R}^N}\Big[ (\mathcal{Q}_{\boldsymbol{\Lambda}}(\boldsymbol{\omega}) \\&&\times {\big( \int_{x_1}^{m_1} \int_{x_2}^{m_2} . . . \int_{x_N}^{m_N}} e^{-i\sum_{k=1}^{N}(b_kt_k\omega_k+c_k\omega_k^2+ e_k\omega_k)} d\boldsymbol{t} \big)\Big] d\boldsymbol{\omega	}
			\\&=&\overline{\sqrt{\frac{\prod_{k=1}^{N}b_k}{(2\pi i)^N}}} e^{\sum_{k=1}^{N}-i(a_kx_k^2+d_kx_k)} \\&& \times \int_{\mathbb{R}^N}\Big[ (\mathcal{Q}_{\boldsymbol{\Lambda}}f)(\boldsymbol{\omega}) e^{-i\sum_{k=1}^{N}(c_k\omega_k^2+ e_k\omega_k)}{\Big( \int_{x_1}^{m} \int_{x_2}^{m} . . . \int_{x_N}^{m}} e^{-i\sum_{k=1}^{N}(b_kt_k\omega_k)} d\boldsymbol{t} \Big)\Big] d\boldsymbol{\omega}
			\\&=&\overline{\sqrt{\frac{\prod_{k=1}^{N}b_k}{(2\pi i)^N}}} e^{\sum_{k=1}^{N}-i(a_kx_k^2+d_kx_k)} \int_{\mathbb{R}^N}  (\mathcal{Q}_{\boldsymbol{\Lambda}}f)(\boldsymbol{\omega}) e^{-i\sum_{k=1}^{N}(c_k\omega_k^2+ e_k\omega_k)} \\&& \times \bigg( \frac{e^{-i(b_1 m_1 \omega_1)}-e^{-i(b_1x_1\omega_1)}}{-ib_1\omega_1} \bigg) \bigg( \frac{e^{-i(b_2 m_2 \omega_2)} -e^{-i(b_2x_2\omega_2)}} {-ib_2\omega_2} \bigg) . .  . \bigg( \frac{e^{-i(b_N m_N \omega_N)} -e^{-i(b_Nx_N\omega_N)}} {-ib_N\omega_N} \bigg) d\boldsymbol{\omega}
			\\&=&\overline{\sqrt{\frac{\prod_{k=1}^{N}b_k}{(2\pi i)^N}}} \int_{\mathbb{R}^N} \frac{1}{\prod_{k=1}^{N}(ib_k \omega_k)} (\mathcal{Q}_{\boldsymbol{\Lambda}}f)(\boldsymbol{\omega}) e^{-i\sum_{k=1}^{N}(a_kx_k^2+ b_kx_k\omega_k+ c_k\omega_k^2+d_kx_k+ e_k\omega_k)} d \boldsymbol{\omega}\\&&+ \overline{\sqrt{\frac{\prod_{k=1}^{N}b_k}{(2\pi i)^N}}} e^{\sum_{k=1}^{N}-i(a_kx_k^2+d_kx_k)} \int_{\mathbb{R}^N} \frac{1}{\prod_{k=1}^{N}(ib_k \omega_k)} (\mathcal{Q}_{\boldsymbol{\Lambda}}f)(\boldsymbol{\omega}) e^{-i\sum_{k=1}^{N}(c_k\omega_k^2+ e_k\omega_k)} 
			\\&& \big( - e^{-i (b_1 x_1 \omega_1 + b_2 m_2 \omega_2 + . . . + b_k x_k \omega_k)} + . . . + (-1)^N e^{-i(b_1 m_1 \omega_1 + b_2 m_2 \omega_2 + . . . + b_k m_k \omega_k)} \big) d\boldsymbol{\omega}.
		\end{eqnarray*}
		Thus
		\begin{eqnarray*}
			&& (B_{\boldsymbol{m}}f)(x) = \mathcal{Q}_{\boldsymbol{-\Lambda}}\bigg( \frac{1}{\prod_{k=1}^{N}(ib_k \omega_k)} (\mathcal{Q}_{\boldsymbol{\Lambda}}f)(\boldsymbol{\omega})\bigg)\\&&+ \overline{\sqrt{\frac{\prod_{k=1}^{N}b_k}{(2\pi i)^N}}} e^{\sum_{k=1}^{N}-i(a_kx_k^2+d_kx_k)} \int_{\mathbb{R}^N} \frac{1}{\prod_{k=1}^{N}(ib_k \omega_k)} (\mathcal{Q}_{\boldsymbol{\Lambda}}f)(\boldsymbol{\omega}) e^{-i\sum_{k=1}^{N}(c_k\omega_k^2+ e_k\omega_k)} 
			\\&& \big( -e^{-i(b_1 x_1 \omega_1 +b_2 m_2 \omega_2 + . . . +b_k x_k \omega_k)} + . . . + (-1)^N e^{-i(b_1 m_1 \omega_1 + b_2 m_2 \omega_2 + . . . + b_k m_k \omega_k)} \big) d\boldsymbol{\omega}.
		\end{eqnarray*}
		As $m \to \infty$, all terms except first in the above expression  $\to 0$.
		Thus, we obtain 		
		\begin{eqnarray*}
			(B f)(\boldsymbol{x}) = \lim_{\boldsymbol{m}\to \infty} (B_m f)(\boldsymbol{x})&=& \mathcal{Q}_{\boldsymbol{-\Lambda}}\bigg( \frac{1}{\prod_{k=1}^{N}(ib_k \omega_k)} (\mathcal{Q}_{\boldsymbol{\Lambda}}f)(\boldsymbol{\omega})\bigg).
		\end{eqnarray*}
		Applying QPFT on both sides gives
		\begin{eqnarray}
			\mathcal{Q}_{\boldsymbol{\Lambda}}(B f)(\boldsymbol{\omega})= \frac{1}{\prod_{k=1}^{N}(ib_k \omega_k)} (\mathcal{Q}_{\boldsymbol{\Lambda}}f)(\boldsymbol{\omega}),
		\end{eqnarray}	
		which leads to
		\begin{eqnarray*}	
			\mathcal{Q}_{\boldsymbol{\Lambda}}(B^2 f)(\boldsymbol{\omega})= \bigg(\frac{1}{\prod_{k=1}^{N}(ib_k \omega_k)}\bigg)^2 (\mathcal{Q}_{\boldsymbol{\Lambda}}f)(\boldsymbol{\omega}).
		\end{eqnarray*}
		Using mathematical induction, we get
		\begin{eqnarray}\label{4.14}
			\mathcal{Q}_{\boldsymbol{\Lambda}}(B^n f)(\boldsymbol{\omega})= \bigg(\frac{1}{\prod_{k=1}^{N}(ib_k \omega_k)}\bigg)^n (\mathcal{Q}_{\boldsymbol{\Lambda}}f)(\boldsymbol{\omega}).		
		\end{eqnarray}
		Combining \ref{nm} and \eqref{4.14}, we get
		\begin{eqnarray*}
			||B^n f||^2_2 &=& ||\mathcal{Q}_{\boldsymbol{\Lambda}} B^n f||^2_2\\ 
			&=& \int_{\mathbb{R}^N}|(\mathcal{Q}_{\boldsymbol{\Lambda}} B^n f)(\boldsymbol{\omega})|^2 d\boldsymbol{\omega}
			\\&=&\int_{\mathbb{R}^N}\Big|\Big(\frac{1}{\prod_{k=1}^{N}(ib_k \omega_k)}\Big)^n (\mathcal{Q}_{\boldsymbol{\Lambda}}f)(\boldsymbol{\omega})\Big|^2d\boldsymbol{\omega}.	
		\end{eqnarray*}
		Let $E= \mathbb{R}^N \setminus \prod_{k=1}^{N}\big(-\frac{\gamma_k}{|b_k|},\frac{\gamma_k}{|b_k|} \big)$. Then $\mathcal{Q}_{\boldsymbol{\Lambda}}f$ vanishes outside of $E$.
		\begin{eqnarray*}
			||B^n f||^2_2 &=& \int_{\mathbb{R}^N} \Big|\Big(\frac{1}{\prod_{k=1}^{N}(ib_k \omega_k)}\Big)^n (\mathcal{Q}_{\boldsymbol{\Lambda}}f)(\boldsymbol{\omega})\Big|^2 d\boldsymbol{\omega}
			\\&\leq&\frac{1}{\prod_{k=1}^{N}\gamma_k^{2n}}\int_{E}|(\mathcal{Q}_{\boldsymbol{\Lambda}}f)(\boldsymbol{\omega})|^2 d\boldsymbol{\omega}
			\\&=& \frac{1}{\gamma^{2n}}||\mathcal{Q}_{\boldsymbol{\Lambda}}f||^2,			
		\end{eqnarray*}
		which gives
		\begin{eqnarray}\label{4.15}
			\lim_{n\to \infty} \sup ||B^n f (x)||^{\frac{1}{n}} &\leq& \frac{1}{\gamma} .		
		\end{eqnarray}
		On the other hand, let $G=\{ \boldsymbol{\omega}\in \mathbb{R}^N: \gamma_k< |b_k\omega_k|<\gamma_k+\epsilon\}$. Then
		\begin{eqnarray}\label{4.16}
			||B^n f||^2_2 &=& \int_{\mathbb{R}^N}\Big|\Big(\frac{1}{\prod_{k=1}^{N}(ib_k \omega_k)}\Big)^n (\mathcal{Q}_{\boldsymbol{\Lambda}}f)(\boldsymbol{\omega})\Big|^2 d\boldsymbol{\omega} \nonumber
			\\&\geq& \int_{G}\Big|\Big(\frac{1}{\prod_{k=1}^{N}(ib_k \omega_k)}\Big)^n (\mathcal{Q}_{\boldsymbol{\Lambda}}f)(\boldsymbol{\omega})\Big|^2 d\boldsymbol{\omega	}\nonumber
			\\ &\geq& \prod_{k=1}^{N}(\gamma_k+\epsilon)^{-2n}\int_{G}|\big(\mathcal{Q}_{\boldsymbol{\Lambda}}f\big)(\boldsymbol{\omega})|^2d\boldsymbol{\omega}. \nonumber
		\end{eqnarray}	
		Consequently, we obtain
		\begin{eqnarray*}
			\lim_{n\to \infty} \inf||B^n f||^{\frac{1}{n}}&\geq& \prod_{k=1}^{N} (\gamma_k+\epsilon)^{-1},
		\end{eqnarray*}	
		which leads to
		\begin{eqnarray}
			\lim_{n\to \infty} \inf||B^n f||^{\frac{1}{n}}&\geq& \prod_{k=1}^{N}\frac{1}{\gamma_k} = \frac{1}{\gamma}.
		\end{eqnarray}
		Combine \ref{4.15} and \ref{4.16} to get
		\begin{eqnarray*}
			\lim_{n\to \infty} ||B^n f||^{\frac{1}{n}}&=& \frac{1}{\gamma}= R.			 
		\end{eqnarray*}
		(Sufficient condition) Suppose $B^n f$ is well-defined and belongs to  $L^2(\mathbb{R}^N)$, $\forall n$ and
		\begin{eqnarray*}
			\lim_{n\to \infty} ||B^n f||^{\frac{1}{n}}=R <\infty,
		\end{eqnarray*}
		where 
		\begin{eqnarray*}
			(Bf)(\boldsymbol{x})&=&e^{-\sum_{k=1}^{N}i(ax_k^2+dx_k)}\int_{x_1}^{\infty} \int_{x_2}^{\infty}. . . \int_{x_N}^{\infty} e^{\sum_{k=1}^{N}i(at_k^2+dt_k)}f(\boldsymbol{t})d\boldsymbol{t}.
		\end{eqnarray*}
		Suppose $\mathcal{Q}_{\boldsymbol{\Lambda}} f$ does not vanish in any neighborhood of origin.
		\\Calculating in similar way, we get
		\begin{eqnarray*}
			(\Delta_{\boldsymbol{x}})^n B^n f (\boldsymbol{x}) = f(\boldsymbol{x}).
		\end{eqnarray*}
		Applying QPFT on both sides, we obtain
		\begin{eqnarray}\label{4.17}
			\mathcal{Q}_{\boldsymbol{\Lambda}}(\Delta_{\boldsymbol{x}}^n B^n f(\boldsymbol{x}))(\boldsymbol{\omega})=(\mathcal{Q}_{\boldsymbol{\Lambda}} f)(\boldsymbol{\omega}).
		\end{eqnarray}
		Replacing $f$ by $B^n f$ in equation \eqref{4.1} gives
		\begin{eqnarray}\label{4.18}
			\mathcal{Q}_{\boldsymbol{\Lambda}}(\Delta_{\boldsymbol{x}}^n B^n f(\boldsymbol{x}))(\boldsymbol{\omega})&=& \prod_{k=1}^{N}(ib_k\omega_k)^n(\mathcal{Q}_{\boldsymbol{\Lambda}} B^n f(\boldsymbol{x}))(\boldsymbol{\omega}).
		\end{eqnarray}	
		Combining \eqref{4.17} and \eqref{4.18} leads to 
		\begin{eqnarray*}
			\prod_{k=1}^{N}(ib_k\omega_k)^n(\mathcal{Q}_{\boldsymbol{\Lambda}} B^n f(\boldsymbol{x}))(\boldsymbol{\omega}) &=&(\mathcal{Q}_{\boldsymbol{\Lambda}} f)(\boldsymbol{\omega}).
		\end{eqnarray*}	
		Hence
		\begin{eqnarray}\label{4.19}
			(\mathcal{Q}_{\boldsymbol{\Lambda}} B^n f(\boldsymbol{x}))(\boldsymbol{\omega}) &=& \frac{(\mathcal{Q}_{\boldsymbol{\Lambda}} f)(\boldsymbol{\omega})}{\prod_{k=1}^{N}(ib_k\omega_k)^n} .
		\end{eqnarray}
		Let $G_0 = \{ \boldsymbol{\omega} \in \mathbb{R}^N: 0 < |\prod_{k=1}^{N} b_k\omega_k| <\epsilon \}$. Then using \ref{nm} and \eqref{4.19}, we obtain
		\begin{eqnarray*}
			||B^n f||^2 &=& ||\mathcal{Q}_{\boldsymbol{\Lambda}} B^n f||^2\\
			&=&\int_{\mathbb{R}^N}\frac{|\mathcal{Q}_{\boldsymbol{\Lambda}} f(\boldsymbol{\omega})|^2}{|\prod_{k=1}^{N}(ib_k\omega_k)|^{2n}} d\boldsymbol{\omega}
			\\&\geq& \int_{G_0}\frac{|\mathcal{Q}_{\boldsymbol{\Lambda}} f(\boldsymbol{\omega})|^2}{|\prod_{k=1}^{N}(b_k\omega_k)|^{2n}} d\boldsymbol{\omega}
			\\&\geq& \frac{1}{\epsilon^{2n}} \int_{G_0} |\mathcal{Q}_{\boldsymbol{\Lambda}} f(\boldsymbol{\omega})|^2 d\boldsymbol{\omega},
		\end{eqnarray*}
		which gives
		\begin{eqnarray*}
			||B^n f||^{\frac{1}{n}} &\geq& \frac{1}{\epsilon} \Big( \int_{G_0} |\mathcal{Q}_{\boldsymbol{\Lambda}} f(\boldsymbol{\omega})|^2 d\boldsymbol{\boldsymbol{\omega}} \Big)^{\frac{1}{2n}}.
		\end{eqnarray*}	
		Thus, it follows that
		\begin{eqnarray*}
			\lim_{n \to \infty}||B^n f||^{\frac{1}{n}} &\geq& \frac{1}{\epsilon}.
		\end{eqnarray*}
		As $\epsilon \to 0$, we obtain
		\begin{eqnarray*}
			\lim_{n\to \infty}||B^n f||^{\frac{1}{n}} = \infty.
		\end{eqnarray*}
		This contradicts our assumption.
		Hence QPFT $(\mathcal{Q}_{\boldsymbol{\Lambda}} f)$ of $f$ vanishes in the neighborhood of origin.
	\end{proof}
	
	\section{Potential Applications}
	\subsection{Solvability of Integral equations}
	As an application of previously defined convolutions, we will now consider some classes of integral equations associated with the convolutions proposed in Theorems \ref{conv1}, \ref{conv2}. This will illustrate one of the possible applications of the above-proved theorems. Here, $\odot$ denotes any of the previously introduced convolutions in \ref{conv1}, \ref{conv2}.
	
	Consider the convolution equation
	\begin{eqnarray}\label{ineq}
		\lambda \phi(\boldsymbol{x}) + (k \odot \phi) (\boldsymbol{x}) = p(\boldsymbol{x}),
	\end{eqnarray}
	where $\lambda \in \mathbb{C}, k, p \in L^1(\mathbb{R}^N)$, and $\phi$ is to be determined there. In equation \eqref{ineq}, when the convolution $\odot$ is taken one of the possibilities \ref{conv1}, or \ref{conv2}, then let us also use $\Omega$ to be corresponding function in 
	\begin{eqnarray}\label{om}
		\{\Omega_1, \Omega_2\},
	\end{eqnarray}
	respectively. We shall now use the notation $S(\boldsymbol{\omega}):= \lambda + \Omega(\boldsymbol{\omega})(\mathcal{Q}_{\boldsymbol{\Lambda}}k)(\boldsymbol{\omega})$.
	\begin{theorem}
		Assume that $S(\boldsymbol{\omega}) \neq 0$ for every $\boldsymbol{\omega}\in \mathbb{R}^N$ and $\frac{\mathcal{Q}_{\boldsymbol{\Lambda}}k}{S} \in L^1(\mathbb{R}^N)$. Then the equation \eqref{ineq} has a solution in $L^1(\mathbb{R})$ if and only if 
		\begin{eqnarray}\label{ineq1}
			\mathcal{Q}_{\boldsymbol{\Lambda}} \big( \frac{\mathcal{Q}_{\boldsymbol{\Lambda}}k}{S} \big) \in L^1(\mathbb{R}).
		\end{eqnarray}
		Moreover, if the condition \eqref{ineq1} holds, then the solution of equation \eqref{ineq} is given in explicit form by $\phi = \mathcal{Q}_{\boldsymbol{\Lambda}}^{-1} \big( \frac{\mathcal{Q}_{\boldsymbol{\Lambda}}k}{S} \big) \in L^1(\mathbb{R}^N)$.
	\end{theorem}
	\begin{proof}
		Necessary condition:
		Suppose that the equation \eqref{ineq} has a solution $\phi \in L^1(\mathbb{R}^N)$. Applying $\mathcal{Q}_{\boldsymbol{\Lambda}}$ to both sides of \eqref{ineq}, we obtain
		\begin{eqnarray*}
			&& (\mathcal{Q}_{\boldsymbol{\Lambda}} \lambda \phi)(\boldsymbol{\omega}) + (\mathcal{Q}_{\boldsymbol{\Lambda}} (k \odot \phi)) (\boldsymbol{\omega}) = (\mathcal{Q}_{\boldsymbol{\Lambda}}p)(\boldsymbol{\omega})
			\\&& \big[ \lambda + \Omega(\boldsymbol{\omega}) (\mathcal{Q}_{\boldsymbol{\Lambda}} k)(\boldsymbol{\omega}) \big] (\mathcal{Q}_{\boldsymbol{\Lambda}}\phi) (\boldsymbol{\omega}) = (\mathcal{Q}_{\boldsymbol{\Lambda}}p)(\boldsymbol{\omega})
			\\&& S(\boldsymbol{\omega}) (\mathcal{Q}_{\boldsymbol{\Lambda}}\phi) (\boldsymbol{\omega}) = (\mathcal{Q}_{\boldsymbol{\Lambda}}p)(\boldsymbol{\omega}).
		\end{eqnarray*}
		Having in mind that $S(\boldsymbol{\omega}) \neq 0$ for all $\boldsymbol{\omega} \in \mathbb{R}$, we get $\mathcal{Q}_{\boldsymbol{\Lambda}} \phi = \frac{\mathcal{Q}_{\boldsymbol{\Lambda}}p}{S} \in L^1(\mathbb{R}^N)$. Then
		\begin{eqnarray*}
			\phi = \mathcal{Q}_{\boldsymbol{\Lambda}}^{-1} \Big(\frac{\mathcal{Q}_{\boldsymbol{\Lambda}}p}{S}\Big) \in L^1(\mathbb{R}^N).
		\end{eqnarray*}
		Sufficient condition: Let $\phi = \mathcal{Q}_{\boldsymbol{\Lambda}}^{-1} \big( \frac{\mathcal{Q}p}{S}\big) \in L^1(\mathbb{R}^N)$. By $\phi \in L^1(\mathbb{R}^N)$, we get $S(\boldsymbol{x}) (\mathcal{Q}_{\boldsymbol{\Lambda}}\phi)(\boldsymbol{x}) = (\mathcal{Q}_{\boldsymbol{\Lambda}} p)(\boldsymbol{x})$.
		Using the factorization identity of the convolution, we obtain
		\begin{eqnarray*}
			\mathcal{Q}_{\boldsymbol{\Lambda}}(\lambda \phi(\boldsymbol{x}) + (k \odot \phi) (\boldsymbol{x})) = \mathcal{Q}_{\boldsymbol{\Lambda}} (p(\boldsymbol{x})).
		\end{eqnarray*}
		By the uniqueness theorem of $\mathcal{Q}_{\boldsymbol{\Lambda}}$, we conclude that $\phi(\boldsymbol{x})$ fulfills \eqref{ineq} for almost every $\boldsymbol{x} \in \mathbb{R}^N$. Hence the proof.
	\end{proof}

	\subsection{Designing of multiplicative filters in the multidimensional QPFT domain}
	In this subsection, we discuss the applications of the new convolution and correlation theorems for the designing of multiplicative filters in the multidimensional QPFT domain
	. These filters are used in various digital signal processing to recover the received signal. In the domain where the signal or noise (or both) are maximally localized, significant performance can be obtained via filtering. Suppose $r(\boldsymbol{x})$, $f(\boldsymbol{x})$ and $n(\boldsymbol{x})$ denote the received signal, the desired signal and the noise, respectively. Then $r(\boldsymbol{x})= f(\boldsymbol{x})+n(\boldsymbol{x})$. The multidimensional QPFT components $\mathcal{Q}_{\boldsymbol{\Lambda}}f (\boldsymbol{\omega})$ and $\mathcal{Q}_{\boldsymbol{\Lambda}}n(\boldsymbol{\omega})$ have no or minimum overlapping. We can design a multiplicative filter in the multidimensional QPFT domain with the convolution and product theorems proposed in this paper. We consider the QPFT- frequency spectrum in the region $[\xi_1, \eta_1] \times , [\xi_2, \eta_2] \times . . . \times [\xi_N, \eta_N]$ of the signal $f(\boldsymbol{x})$. Let $g(\boldsymbol{x})$ be the filter impulse response. Then $\mathcal{Q}_{\boldsymbol{\Lambda}}g (\boldsymbol{\omega})$ is constant over $[\xi_1, \eta_1] \times , [\xi_2, \eta_2] \times . . . \times [\xi_N, \eta_N]$ which either vanishes or decay rapidly outside that region.
	The output signal can be expressed as
	\begin{eqnarray*}
		r_{out}(\boldsymbol{x})= (r_{in}\odot g)(\boldsymbol{x}).
	\end{eqnarray*}
	With this, we get 
	\begin{eqnarray*}
		r_{out}(\boldsymbol{x})
		&=&\mathcal{Q}_{\boldsymbol{-\Lambda}}\big[\mathcal{Q}_{\boldsymbol{\Lambda}} (r_{in}\odot g)\big]
		\\&=& \mathcal{Q}_{\boldsymbol{-\Lambda}} \big[ \Omega \mathcal{Q}_{\boldsymbol{\Lambda}} (r_{in})(\boldsymbol{\omega}) \mathcal{Q}_{\boldsymbol{\Lambda}}g(\boldsymbol{\omega}) \big](\boldsymbol{x}),
	\end{eqnarray*}
	where $\Omega$ is as in \eqref{om}, whose multidimensional QPFT has a finite support in the region $[\xi_1, \eta_1] \times , [\xi_2, \eta_2] \times . . . \times [\xi_N, \eta_N]$, so it preserves that part of the spectrum of $f(\boldsymbol{x})$ over $[\xi_1, \eta_1] \times , [\xi_2, \eta_2] \times . . . \times [\xi_N, \eta_N]$.
	\\i.e., 
	\begin{eqnarray*}
		\mathcal{Q}_{\boldsymbol{\Lambda}} [r_{out}(\boldsymbol{x})](\boldsymbol{\omega}) = \mathcal{Q}_{\boldsymbol{\Lambda}}f(\boldsymbol{\omega}), ~~~~~~~~~~~\boldsymbol{\omega} \in [\xi_1, \eta_1] \times , [\xi_2, \eta_2] \times . . . \times [\xi_N, \eta_N].
	\end{eqnarray*}
	
	\section{Conclusion}
	In this article, we have extended the QPFT to the multidimensional domain and established some key properties, including Plancherel and inversion theorems. We further explored three different convolutions and a correlation, and their corresponding theorems for multidimensional QPFT, generalizing convolutions of single variable QPFT. Furthermore, we derived Boas theorem associated with multidimensional QPFT. Besides the theoretical developments, we have also discussed the applications of the proposed convolutions in solving integral equations and multiplicative filter designing, which can be used in optics and signal processing for signal recovery. 
	
	\section*{Declaration of competing interest}
	The authors declare that they have no known competing financial interests or personal relationships that could have appeared to influence the work reported in this paper.
	\\\\
	\textbf{Note}: The manuscript has no associated data.
	
	\bibliographystyle{amsplain}

\begin{thebibliography}{10}
		
		\bibitem{intt} Debnath, L., Bhatta, D., Integral Transforms and their applications. In Chapman and Hall/CRC eBooks. (2014).
		
		\bibitem{intt1} Davies, B., Integral Transforms and their Applications. In Applied mathematical sciences. (1985). 
		
		\bibitem{1} Bracewell, R, N.: The Fourier transform and its applications. Vol. 31999. McGraw-Hill, New York (1986).
		
		\bibitem{frac1} Namias, V., The Fractional order Fourier transform and its application to quantum mechanics. IMA J. Appl. Math. 25(3)(1980): 241–265. 
		
		\bibitem{frac2} Almeida, L., The fractional Fourier transform and time-frequency representations.IEEE Trans. Signal Process. 42(11)(1994): 3084–3091. 
		
		\bibitem{frac3} Ozaktas, H. M., Zalevsky, Z., Kutay, M. A., The Fractional Fourier Transform with Applications in Optics and Signal Processing. Wiley. (2001).
		
		\bibitem{frac} Zayed, A. I., Fractional Integral transforms: Theory and Applications. CRC Press. (2024).	
		
		\bibitem{frac4} Almeida, L. B.: Product and Convolution Theorems for the Fractional Fourier Transform. IEEE Signal Process. Lett. 4(1)(1997): 15–17.				
		
		\bibitem{lct} Collins, S. A., Lens-system diffraction integral written in terms of matrix optics, J. Optical Soc. Amer., 60(9)(1970): 1168–1177.
		
		\bibitem{lct1} Moshinsky, M., Quesne,  C., Linear canonical transformations and their unitary representations, Math Phys. 12(1971): 1772–1780.
		
		\bibitem{lct2} Goel, N., Singh, K.: Modified correlation theorem for the linear canonical transform with representation transformation in quantum mechanics. Signal Image Video Process. 8(3)(2013): 595–601. 
		
		\bibitem{lct3} Deng, B., Tao, R., Wang, Y.: Convolution theorems for the linear canonical transform and their applications. Sci. China Inf. Sci. 49(5)(2006): 592–603.	
				
		\bibitem{lpc} Castro, L. P., Minh, L. T. and Tuan, N. M., New Convolutions for Quadratic-Phase Fourier Integral Operators and their Applications, Mediterr. J. Math. 15(13)(2018):13.
		
		\bibitem{lmm} Castro, L. P., Haque, M. R.,   Murshed, M. M., Saitoh,S. and Tuan, N. M.,  Quadratic Fourier Transforms,
		Ann. Funct. Anal. 5(1)(2014): 10–23.
		
		\bibitem{sai} Saitoh, S., Theory of reproducing kernels: Applications to approximate solutions of bounded linear operator functions on Hilbert spaces. Am. Math. Soc. Trans. Ser. 230(2010): 107–134.
		
		\bibitem{ta} Sharma, P. B., Prasad, A.: Convolution and product theorems for the quadratic-phase Fourier transform. Georgian Math. J. 29(4)(2022): 595–602.
		
		\bibitem{pb} Prasad, A., Sharma, P.B.: The quadratic‐phase Fourier wavelet transform. Math. Methods Appl. Sci. 43(4)(2019): 1953–1969. 
		
		\bibitem{sp1} Sharma, P. B.,  Prasad, A., Abelian Theorems for Quadratic-Phase Fourier Wavelet Transform. Proc. Natl. Acad. Sci. India Sect. A. 93(1)(2022): 75–83. 
		
		\bibitem{sp2} Sharma, P. B., Prasad, A., Composition of Quadratic-Phase Fourier Wavelet Transform. Int. J. Appl. Comput. Math. 8(3)(2022): 90.
		
		
		\bibitem{yg} Kumar, M and Pradhan, T., Quadratic-phase Fourier transform of tempered distributions and pseudo-differential operators, Integral Transforms Spec. Funct. 33(6)(2021): 449–465.
		
		\bibitem{tml} Lai, T. M., Modified Ambiguity Function and Wigner Distribution Associated With Quadratic-Phase Fourier Transform, J. Fourier Anal. Appl. 30(6)(2024):6
		
		
		\bibitem{var} Varghese, S. and Kundu, M., Spectrum-Related Theories in the Framework of Quadratic Phase Fourier Transform. Math. Meth. Appl. Sci. (2024). $https://doi.org/10.1002/mma.10642$
		
		\bibitem{bhat} Bhat, M. Y., Dar, A. H., Urynbassarova, D., Urynbassarova, A., Quadratic-phase wave packet transform. Optik. 261(2022): 169120.
		
		\bibitem{lb1} Kundu, M. and Prasad, A.: Convolution, correlation and spectrum of functions associated with linear canonical transform. Optik. 249(2022):168256. 
		
		\bibitem{dw4} Wei, D., Ran, Q., Li, Y., Ma, J. and Tan, L., A Convolution and Product Theorem for the Linear Canonical Transform, IEEE Signal Process. Lett. 16(10)(2009): 853–856.
		
		\bibitem{dw3} Wei, D., Ran, Q. and Li, Y., A Convolution and Correlation Theorem for the Linear Canonical Transform and Its Application, Circuits Systems Signal Process. 31(1)(2011): 301–312. 
		
		\bibitem{dw2} Wei, D., New product and correlation theorems for the offset linear canonical transform and its applications, Optik, 164(2018):243–253.
		
		\bibitem{dwy} Wei, D. and Li, Y. M., Generalized wavelet transform based on the convolution operator in the linear canonical transform domain, Optik. 125(16)(2014): 4491–4496.
		
		\bibitem{dwq} Wei, D., Ran, Q. and Li, Y., A Convolution and Correlation Theorem for the Linear Canonical Transform and Its Application, Circuits Systems Signal Process. 31(1)(2011): 301–312.
		
		\bibitem{dw1} Wei, D., Novel convolution and correlation theorems for the fractional Fourier transform, Optik. 127(7)(2016): 3669–3675.
		
		\bibitem{bl1} Feng, Q. and Li, B., Convolution and correlation theorems for the two‐dimensional linear canonical transform and its applications, IET Signal Process. 10(2)(2016): 125–132.
		
		\bibitem{bl2} Guo, Y. and Li, B. Z., The linear canonical wavelet transform on some function spaces,  Int. J. Wavelets Multiresolut. Inf. Process. 16(1)(2018): 1850010.
		
		\bibitem{boas} Boas, R. P, Jr., Some theorems on Fourier transforms and conjugate trigonometric integrals, Trans. Am. Math. Soc. 40(2)(1936): 287-308.
		
		\bibitem{tuan}  Tuan, V. K., Spectrum of signals. J. Foureir Anal. Appl. 7(3)(2001): 319–323.
		
		\bibitem{inte} Banerjea, S., Mandal, B. N., Integral equations and integral transforms. Springer Nature. (2023)
		
		\bibitem{inte1} Rahman, M., Integral equations and their applications. WIT Press. (2007). 
		
		\bibitem{inte2} Jerri, A. J.  Introduction to Integral Equations with Applications. John Wiley and Sons. (1999).
		
		\bibitem{multis0} Chen, N. T., The past, present, and future of image and multidimensional signal processing. IEEE Signal Processing Magazine, 15(2)(1998): 21–58.
		
		\bibitem{multis3} Bose, N. K., Multidimensional systems and signal processing. Multidimens. Syst. Signal Process. 17(4)(2006), 297–298. 
		
		\bibitem{multis1} Maragos, P., Schafer, R., Morphological systems for multidimensional signal processing. Proceedings of the IEEE, 78(4)(1990): 690–710.
		
		\bibitem{multis2} Bose, N. K., Applied Multidimensional Systems Theory. In Springer eBooks. (2016). 	
		
		\bibitem{Zayed} Zayed, A. I., Two-dimensional fractional Fourier transform and some of its properties. Integral Transforms Spec Funct. 29(2018): 553–570.
		
		\bibitem{roop} Kamalakkannan, R., Roopkumar, R.: Multidimensional fractional Fourier transform and generalized fractional convolution. Integral Transforms Spec. Funct. 31(1)(2020): 152-65.
		
		\bibitem{Kundu} Kundu, M., Prasad, A., Verma, R.: Multidimensional linear canonical transform and convolution. J. Ramanujan Math. Soc. 37(2)(2022): 159-171.	
		
		\bibitem{ahm} Ahmad, O., Achak, A., Sheikh, N. A., Warbhe, U., Uncertainty principles associated with multi-dimensional linear canonical transform, Int. J. Geom. Methods Mod. Phys. (19)(2)(2022): 2250029. 
		
		\bibitem{castro} Castro, L. P., Guerra, R., Multidimensional quadratic-phase Fourier transform and its uncertainty principles. Constructive Mathematical Analysis, 8(1)(2025): 15–34. 
		
		\bibitem{kumar} Kumar, M., Bhawna, N., An $n-$ dimensional pseudo-differential operator involving quadratic phase Fourier transform and applications in quantum mechanics. J. Pseudo-Differ. Oper. Appl. 16(2)(2025): 28.
		
		\bibitem{Rudin} Rudin, W., Real and Complex Analysis. 3rd ed. McGraw-Hill Inc., New York (1987).
	\end{thebibliography}

\end{document}